\documentclass[12pt]{amsart}

\usepackage{amsmath,amssymb,amsfonts,latexsym}

\usepackage{mathrsfs}%pour les caracteres en calligraphie
\usepackage[dvips]{geometry}
\usepackage{array}
\usepackage{color}
\usepackage[frenchb, english]{babel}
\usepackage[T1]{fontenc}

\usepackage{epsfig, graphicx} %Work with figures.
\usepackage{multirow}

\newtheorem{theorem}{Theorem}[section]
\newtheorem{lemma}[theorem]{Lemma}
\newtheorem{corollary}[theorem]{Corollary}
\newtheorem{Definition}[theorem]{Definition}
\newtheorem{Remark}[theorem]{Remark}
\newtheorem{Example}[theorem]{Example}

\newcommand{{\ind}}{\operatorname{ind}}
\newcommand{{\ad}}{\operatorname{ad}}
\newcommand{{\card}}{\operatorname{Card}}
\newcommand{{\rk}}{\operatorname{rk}}
\newcommand{{\id}}{\operatorname{id}}
\newcommand{{\trp}}{\mathfrak{p}_{\pi', E}}
\newcommand{{\liea}}{\mathfrak{a}}
\newcommand{{\lieg}}{\mathfrak{g}}
\newcommand{{\liep}}{\mathfrak{p}}
\newcommand{{\lieh}}{\mathfrak{h}}
\newcommand{{\lien}}{\mathfrak{n}}
\newcommand{{\liel}}{\mathfrak{l}}
\newcommand{{\lieb}}{\mathfrak{b}}
\newcommand{{\liesl}}{\mathfrak{sl}}
\newcommand{{\bbN}}{\mathbb{N}}
\newcommand{{\bbC}}{\mathbb{C}}
\newcommand{{\ch}}{\operatorname{ch}}

\begin{document}


\title[Poisson centre of maximal parabolic subalgebras]
{Polynomiality for the Poisson centre of truncated maximal parabolic subalgebras.}

\author{Florence Fauquant-Millet, Polyxeni Lamprou }
%\lastname{Fauquant-Millet}

%\msc {16W22, 17B22, 17B35}

\date{}

\address{Florence Fauquant-Millet\\
Universit\'e de Lyon\\
 Universit\'e Jean Monnet\\
 Saint-Etienne\\
 ICJ UMR 5208\\
  F-42023, Saint-Etienne\\
florence.millet@univ-st-etienne.fr
}
\address{Polyxeni Lamprou\\
polyxeni.lamprou@googlemail.com}

\def\a{\mathfrak a}
\def\g{\mathfrak g}
\def\h{\mathfrak h}
\def\p{\mathfrak p}
\def\n{\mathfrak n}
\def\b{\mathfrak b}
\def\l{\mathfrak l}
\def\o{\mathfrak o}
\def\ep{\varepsilon}
\def\al{\alpha}
\def\s{\mathfrak s}
\def\l{\mathfrak l}
\def\ep{\varepsilon}
\def\q{\mathfrak q}
\def\m{\mathfrak m}

\begin{abstract}
We study the Poisson centre of  truncated maximal parabolic subalgebras of a simple Lie algebra of type B, D or ${\rm E}_6$. In particular we show that this centre is a polynomial algebra and compute the degrees of its generators.
In roughly half of the cases the polynomiality of the Poisson centre was already known by a completely different method.
For the rest of the cases, our approach is to construct an {\it algebraic slice} in the sense of Kostant given by an {\it adapted pair} and the computation of an improved upper bound for the Poisson centre. 

\end{abstract}

\maketitle

\section{Introduction.}\label{intro}~~

The base field $k$ is assumed to be algebraically closed of characteristic zero.

In this paper we continue our study on the Poisson semicentre of maximal parabolic subalgebras of a simple Lie algebra over $k$, that we initiated in~\cite{FL}.

Let $\p$ be a parabolic subalgebra of a semisimple Lie algebra $\g$ over $k$.
Recall that the semicentre $Sy(\p)$ of the symmetric algebra $S(\p)$ of $\p$ is the vector space generated by the semi-invariants of $S(\p)$ under the adjoint action of $\p$. 
When $S(\p)$ is equipped with its natural Poisson structure, the semicentre $Sy(\p)$ of $\p$ coincides with the Poisson semicentre of $S(\p)$ (of $\p$ for short).
The algebra of invariants $S(\p)^{\p}$ of $S(\p)$ under the adjoint action of $\p$ will be denoted by $Y(\p)$. Again the algebra $Y(\p)$ coincides with the Poisson centre of $S(\p)$ (of $\p$ for short) when $S(\p)$ is equipped with its natural Poisson structure.
 By a result of \cite[Satz 6.1]{BGR}, since $\p$ is algebraic, there is a canonically defined algebraic subalgebra $\p_{\Lambda}$ of $\p$, called  the canonical truncation of $\p$, such that $Sy(\p)=Y(\p_{\Lambda}):=S(\p_{\Lambda})^{\p_{\Lambda}}$.
Actually $\p_{\Lambda}$ is the largest subalgebra of $\p$ which vanishes on the weights of $Sy(\p)$. One has also trivially that $Sy(\p_{\Lambda})=Y(\p_{\Lambda})$.

Recall that the Poisson centre $Y(\p)$ of $\p$ is reduced to $k$, when $\p$ is not equal to $\g$ and $\g$ simple (see for example~\cite[7.9]{J6} or~\cite[Chap. I, Sec. B, 8.2 (iv)]{F}), whereas the Poisson semicentre $Sy(\p)$ of $\p$ is never reduced to scalars by~\cite{D1}.

By~\cite{FJ2} - see also~\cite{J6},~\cite{J7} in the more general case of biparabolic (seaweed) subalgebras - we know that $Sy(\p)$
is lower and upper bounded, up to gradations, by polynomial algebras $\mathcal A$ and $\mathcal B$ respectively, having the same number of generators. The weight of each generator of $\mathcal A$  may either be equal or  be  the double of the weight of the corresponding generator of $\mathcal B$. Moreover, it was shown that the coincidence of the formal characters ${\rm ch}\,\mathcal A$ and ${\rm ch}\,\mathcal B$  of these bounds is a sufficient condition for the polynomiality of $Sy(\p)$. The coincidence of ${\rm ch}\,\mathcal A$ and ${\rm ch}\,\mathcal B$  occurs often, for instance when $\g$ is simple of type ${\rm A}$ or ${\rm C}$ and $\p$ is any parabolic subalgebra of $\g$.

However, the coincidence of ${\rm ch}\,\mathcal A$ and ${\rm ch}\,\mathcal B$  is not a necessary condition for the polynomiality of the Poisson semicentre and indeed there are examples where they do not coincide but the Poisson semicentre is polynomial, for example in the Borel case~\cite{J1}.

Since $Sy(\p_{\Lambda})=Y(\p_{\Lambda})$, the field $C(\p_{\Lambda}):=({\rm Fract}\, S(\p_{\Lambda}))^{\p_{\Lambda}}$ of invariant fractions of $S(\p_{\Lambda})$ is equal to the field of fractions ${\rm Fract}\,(Y(\p_{\Lambda}))$ of $Y(\p_{\Lambda})$,  as each semi-invariant of ${\rm Fract}\,S(\a)$ is a quotient of two semi-invariants of $S(\a)$, for any finite dimensional Lie algebra $\a$ by~\cite{D1} or~\cite[Chap. I, Sec. B, 5.11, 5.12]{F}. Hence the polynomiality of $Sy(\p)=Y(\p_{\Lambda})$ implies that the field of invariant fractions $C(\p_{\Lambda})$ is a purely transcendental extension of the base field $k$ and by~\cite[Thm. 66]{O} so is the field of invariant fractions $C(\p)$, since there exists a set of algebraically independent generators of $Sy(\p)$ formed by weight vectors, that is by semi-invariants of $S(\p)$. This allows us to answer positively Dixmier's fourth problem for such parabolic subalgebras, namely whether the field of invariant fractions is a purely transcendental extension of the base field, for any finite dimensional Lie algebra. However the polynomiality of the Poisson centre $Y(\p_{\Lambda})$ is a much stronger result.

Recently, several authors have been interested in the question of polynomiality of the Poisson centre of non-reductive algebraic Lie algebras; parabolic and biparabolic (seaweed) subalgebras of a simple Lie algebra $\g$ over $k$ were studied in~\cite{FJ2},~\cite{FJ3},~\cite{J6},~\cite{J7} and some particular semi-direct products were studied in~\cite{PPY1}, \cite{PPY2}, \cite{PPY3}, \cite{Y1}, \cite{Y2}, where polynomiality of the Poisson centre was shown. In~\cite{O} the author gives necessary and sufficient conditions for the Poisson centre or semicentre of certain finite dimensional Lie algebra  to be polynomial.

So far, only one counterexample to the polynomiality of the Poisson semicentre of a biparabolic subalgebra $\p$ is known, namely when $\g$ is of type ${\rm E}_8$ and $\p$ is the maximal parabolic subalgebra of $\g$, whose canonical truncation coincides with the centralizer of the highest root vector of $\g$~\cite{Y}.

In~\cite{FL} we studied $Sy(\p)$ for $\p$ a maximal parabolic subalgebra of a simple Lie algebra $\g$, when the lower and upper bounds ${\rm ch}\,\mathcal A$ and ${\rm ch}\,\mathcal B$ coincide (hence $Sy(\p)$ is polynomial) and we constructed slices for the coadjoint action, extending the Kostant Slice Theorem~\cite[Thm. 0.10]{K}.

In this paper we study the remaining cases for $\g$ simple of type ${\rm B,\, D}$ and ${\rm E_6}$ and we deduce the polynomiality of the Poisson semicentre $Sy(\p)$ by constructing slices for the coadjoint action and computing an ``improved upper bound'' (see below).

The slices we constructed in~\cite{FL} were given by {\it adapted pairs} (see Section~\ref{standardnotation}) for the canonical truncations $\p_{\Lambda}$ of the parabolic subalgebras $\p$ that we studied. 
In this paper we construct adapted pairs for the remaining cases mentioned above.

Adapted pairs play the role of principal $\mathfrak{sl}_2$-triples in the non-reductive case and were introduced in~\cite{JL}.
They give an improved upper bound $\mathcal{B}'$ for the character of $Sy(\mathfrak p)=Y(\p_{\Lambda})$ \cite{J6bis}. When this bound is attained, in particular when it coincides with the character of the lower bound $\mathcal{A}$ mentioned above, polynomiality of $Sy(\p)$ follows and the adapted pair gives an algebraic slice (in the sense of~\cite[7.6]{J8}) also called a Weierstrass section in~\cite{FJ4}, extending the Kostant Slice Theorem~\cite[Thm. 0.10]{K} to non-reductive Lie algebras.
By~\cite{FJ4}, this Weierstrass section is also an affine slice for the coadjoint action (in the sense of~\cite[7.3]{J8}).

Some particular cases had already been studied by other authors and different methods. For example, it was shown in~\cite{PPY} that for all maximal parabolic subalgebras $\p$ whose canonical truncation is the centralizer of the highest root vector of the simple Lie algebra (except in type ${\rm E}_8$, where we have Yakimova's counterexample), the Poisson semicentre $Sy(\p)$ is a polynomial algebra over $k$.

Furthermore, Heckenberger \cite{H} showed by computer calculations that in type ${\rm B}_n$, $2\le n\le 4$, the Poisson semicentre $Sy(\p)$ is polynomial for all parabolic subalgebras $\p$.

In~\cite{TY} an affine slice for the coadjoint action of $\p$ was constructed for some non truncated biparabolic subalgebras $\p$ of a simple Lie algebra, which gave a positive answer to Dixmier's fourth problem for $C(\p)$. These biparabolic subalgebras $\p$ do not coincide with the maximal parabolic subalgebras we are interested in.

Below, labeling of simple roots follows Bourbaki \cite[Planches I-IX]{BOU}.

Adapted pairs need not exist for all truncated parabolic subalgebras and are very hard to construct in general. One may hope to construct such pairs when the truncated Cartan subalgebra - that is, the subalgebra of the Cartan subalgebra, which is contained in the canonical truncation of the parabolic subalgebra we consider - is large enough, as it happens when $\g$ is of type $\rm A$ or when the parabolic subalgebra $\p$ is maximal; however, we showed that even in these favourable cases adapted pairs may not exist, as it happens for example when $\g$ is of type ${\rm F}_4$ and $\p$ is the maximal parabolic subalgebra corresponding to $\pi'=\{\alpha_1,\, \alpha_2,\,\alpha_4\}$ \cite[Sect. 10]{FL}. In type ${\rm A}$ adapted pairs were constructed for all truncated biparabolic subalgebras in~\cite{J5}.

When the parabolic subalgebra $\p$ is maximal associated to $\pi'=\pi\setminus\{\alpha_s\}$ where $\pi$ is a set of simple roots $\alpha_i,\, 1\le i\le n,$ in $\g$ and $\g$ is simple of type ${\rm B}_n$ or ${\rm D}_n$, the bounds ${\rm ch}\,\mathcal A$ and ${\rm ch}\,\mathcal B$ for $Sy(\p)$ coincide exactly when $s$ is odd (in type ${\rm D}_n$, $n\ge 4$, under the restriction $s\neq n-1$; additionally, when $s=n-1$ and $s$ even, and finally in type ${\rm D}_4$ for all $s$ except for $s=2$; in type ${\rm B}_n$, $n\ge 2$, also for $n=s=2$ and $n=s=4$).

In this paper we give an adapted pair for the rest of the truncated maximal parabolic subalgebras in type ${\rm B}$ and ${\rm D}$. In particular, we prove a lemma of non-degeneracy (Lemma~\ref{non-degeneracy}) which is a non-obvious generalization of ~\cite[Lemma 5]{FL}.

From the case ${\rm D}_6$, $s=6$, we also deduce in Section~\ref{E7} an adapted pair for the truncated maximal parabolic subalgebra of $\g$ of type ${\rm E}_7$ corresponding to $\pi'=\pi\setminus\{\alpha_3\}$.

Finally we construct in Section~\ref{secE6} an adapted pair for the truncated maximal parabolic subalgebras $\p_{\Lambda}$ in a simple Lie algebra of type ${\rm E}_6$, when the bounds $\mathcal A$ and $\mathcal B$ do not coincide, that is when $s=1,\,6$ (for $s=2$ an adapted pair was already constructed in~\cite{J6bis}).

Then we compute the improved upper bound $\mathcal B'$ (Lemmas~\ref{improvboundB},\,\ref{improvboundD},~\ref{improvboundDe} and Sections~\ref{E7} and~\ref{secE6}) and we show that it is attained and hence the Poisson centre $Y(\p_{\Lambda})$ of $\p_{\Lambda}$ is polynomial (Theorems~\ref{conB}, \ref{conD},~\ref{conDe},~\ref{conE7} and \ref{conE6}). We deduce that for all such maximal parabolic subalgebras $\p$, Dixmier's fourth problem is true for $C(\p)$.
Furthermore, as in~\cite{FL} we obtain an algebraic and an affine slice for the dual of $\p_{\Lambda}$.
\smallskip

{\bf Acknowledgements.}  We would  like to thank A. Joseph for many fruitful discussions on adapted pairs and for his interest in our work. We are also grateful to A. Ooms for enlightening exchange of ideas on the polynomiality of the Poisson semicentre. Part of these results were presented by the first author in the Seminar at the Weizmann Institute of Science in Israel in April 2016 and in the Conference ``Algebraic Modes of Representations and Nilpotent Orbits : the Canicular Days'', celebrating A. Joseph's 75th birthday, in Israel in July 2017 and by the second author in the Conference ``Representation Theory in Samos'' in Greece in July 2016.

\section{Preliminaries.}\label{standardnotation}

Let $\mathfrak{g}$ be a finite dimensional
semisimple Lie algebra over $k$ and $\mathfrak{h}$ a fixed
Cartan subalgebra of $\mathfrak{g}$.

Let $\Delta$ be the root system of
$\mathfrak{g}$ with respect to $\mathfrak{h}$, $\pi$ a chosen set of
simple roots, $\Delta^+$ (resp. $\Delta^-$)  the set of positive (resp. negative) roots. We adopt the labeling of \cite[Planches I-IX]{BOU} for the simple roots in $\pi$.

For any $\alpha\in \Delta$, let $\mathfrak g_\alpha$ denote the corresponding root space of $\mathfrak g$ and fix a nonzero vector $x_{\alpha}$ in $\g_{\alpha}$. Then $\mathfrak{g}
= \mathfrak{n}\oplus \mathfrak{h}\oplus \mathfrak{n}^-$, where
$\mathfrak{n}=\bigoplus_{\alpha\in \Delta^+}\mathfrak g_\alpha$ and $\mathfrak{n}^-=\bigoplus_{\alpha\in \Delta^-}\mathfrak g_\alpha$.
For all $\alpha\in\pi$, denote by $\alpha^\vee$ the corresponding coroot.
For any subset $A$ of $\Delta$, set $\g_A=\bigoplus_{\alpha\in A}\g_{\alpha}$.

For any subset $\pi^{\prime}$ of $\pi$, let
$\Delta_{\pi'}$ be the subset of roots in $\Delta$ generated by $\pi'$ and $\Delta_{\pi'}^+, \, \Delta_{\pi'}^-$ the sets of positive and negative roots in $\Delta_{\pi'}$ respectively.

One defines the standard parabolic
subalgebra $\mathfrak{p}_{\pi^{\prime}}$ associated to $\pi'$ to be the algebra $\mathfrak{p}_{\pi'} = \lien \oplus
\lieh \oplus \lien_{\pi'}^-$ where $\lien^-_{\pi'}=\bigoplus_{\alpha\in \Delta^-_{\pi'}}\g_\alpha$. Its opposed algebra then is $\mathfrak{p}_{\pi'}^- = \lien^- \oplus
\lieh \oplus \lien_{\pi'}$, with $\lien_{\pi'}$ defined similarly. The dual space $\mathfrak p_{\pi'}^*$ identifies with $\liep_{\pi'}^-$ via the Killing form $K$ on $\g$.

We denote by $W_{\pi'}$ the Weyl group associated to $\pi'$ and by $r_{\gamma}$, for $\gamma\in\Delta_{\pi'}$ the reflection with respect to $\gamma$. Then $W_{\pi'}$ is the subgroup of the Weyl group $W$ of $(\g,\,\h)$, generated by $r_{\gamma}$, for all $\gamma\in\Delta_{\pi'}$.

Let $\a$ be a finite dimensional Lie algebra over $k$. The semicentre $Sy(\a)$ of its symmetric algebra $S(\a)$ (of $\a$ for short) is defined to be the vector space spanned by the semi-invariants under the adjoint action of $\a$ that is, $Sy(\a)=\bigoplus_{\lambda\in\a^*}S(\a)_{\lambda}$ where $S(\a)_{\lambda}=\{s\in S(\a)\mid\forall x\in\a,\,(\ad x) s=\lambda(x) s\}$. It is a subalgebra of $S(\a)$. When $S(\a)_{\lambda}\neq\{0\}$, $\lambda$ is called a weight of the semicentre $Sy(\a)$. Let $\Lambda(\a)$ denote the set of weights of $Sy(\a)$. 

When $\a=\p_{\pi'}$, the set $\Lambda(\p_{\pi'})$ of weights of $Sy(\p_{\pi'})$ may be identified with a subset of $\h^*$ and we have also that $Sy(\p_{\pi'})$ is equal to the algebra of invariants $S(\p_{\pi'})^{\p_{\pi'}'}$ of $S(\p_{\pi'})$ under the adjoint action of the derived algebra $\p'_{\pi'}$ of $\p_{\pi'}$.
 
Equip $S(\a)$ with its natural Poisson structure coming from the Lie bracket on $\a$. The Poisson centre $Y(\a)$ of $\a$ is the centre of $S(\a)$ for this structure and it is also the set of the invariants in $S(\a)$ under the adjoint action of $\a$, that is $Y(\a)=S(\a)_0$. It is an algebra contained in the semicentre $Sy(\a)$ of $S(\a)$. Again $Sy(\a)$ is also the Poisson semicentre of $S(\a)$ for its natural Poisson structure.

If $\a$ is algebraic, there is an algebraic subalgebra of $\a$, called the canonical truncation of $\a$, $\a_{\Lambda}=\cap_{\lambda\in\Lambda(\a)} {\rm ker}\, \lambda$,  such that
$Sy(\a)=Sy(\a_{\Lambda})=Y(\a_{\Lambda})$ ~\cite[Satz 6.1]{BGR}. The algebra $\a_{\Lambda}$ is an ideal of $\a$ containing the derived subalgebra of $\a$.

The index of $\a$, denoted by $\ind{\a}$, is the minimal dimension of a stabilizer $\a^f$ for $f\in\a^*$. When $\a$ is algebraic, the index of $\a$ is also equal to the minimal codimension of a coadjoint orbit in $\a^*$
~\cite[1.11.3]{D}.

An element $y\in\a^*$ is called regular in $\a^*$ if  its stabilizer $\a^y$ is of minimal dimension (equal to $\ind\a$).

Let $\pi'\subset \pi$. Since $\p_{\pi'}$ is algebraic, the canonical truncation  $\p_{\pi',\,\Lambda}$ of $\p_{\pi'}$, which we recall is defined to be the largest subalgebra of $\p_{\pi'}$ that vanishes on the weights of $Sy(\p_{\pi'})$,
has the property that the Poisson centre $Y(\p_{\pi',\,\Lambda})$ is equal to the Poisson semicentre $Sy(\p_{\pi',\,\Lambda})$ and also equal to $Sy(\p_{\pi'})$.

The canonical truncation of $\p_{\pi'}$ was given explicitly in \cite{FJ3}. It is of the form $\liep_{\pi', \Lambda} = \n\oplus \h_{\Lambda}\oplus \n^-_{\pi'}$ where $\h_{\Lambda}$ is a subalgebra of $\h$ called the truncated Cartan subalgebra (this is the largest subalgebra of $\h$ which vanishes on the set of weights $\Lambda(\p_{\pi'})$ of $Sy(\p_{\pi'})$).

The Gelfand-Kirillov dimension of $Y(\p_{\pi',\,\Lambda})$ is equal to the index of $\p_{\pi',\,\Lambda}$. For more details, see~\cite[2.4, 2.5, B.2]{FJ3}.

Let $\h^{\prime}\subset\h$ be the Cartan subalgebra of the Levi factor of $\p_{\pi'}$. When $\pi'=\pi\setminus\{\alpha_s\}$, then $\h_{\Lambda}=\h^{\prime}$ that is, $\h_{\Lambda}$ is the vector space over $k$
generated by all $\alpha^{\vee}$ with $\alpha\in\pi'$.

For convenience, we replace the truncated parabolic subalgebra $\p_{\pi',\,\Lambda}$ by its opposed algebra $\p^-_{\pi',\,\Lambda}$ (that is, the canonical truncation of the opposed algebra $\p_{\pi'}^-$).
From now on, we denote it simply by $\p$.

For any $\h$-module $M=\bigoplus_{\nu\in\h^*} M_{\nu}$ with finite dimensional weight spaces $M_{\nu}:=\{m\in M\,\mid\,\forall\, h\in \h,\,h.m=\nu(h) m\}$, we may define its formal character by $\ch M=\sum_{\nu\in\h^*}\dim M_{\nu} e^{\nu}.$ Given two such  $\h$-modules $M$ and $M'$
write $\ch M\le \ch M'$ if $\dim\,M_{\nu}\le\dim\,M'_{\nu}$ for all $\nu\in\h^*$~\cite[2.8]{J7}.

Here we recall the formal characters $\ch\mathcal A$ and $\ch\mathcal B$ of the lower and the upper bounds  mentioned in the introduction for $\ch Y(\p)$ given in~\cite[Thm. 6.7]{J7}.

Let $E(\pi')$ be the set of $\langle {\bf ij}\rangle$-orbits of $\pi$, where ${\bf i}$ and ${\bf j}$ are the involutions of $\pi$ defined  for example in~\cite[2.2]{FL}. 
For the reader's convenience, we give below their definition.

Let $w_0$ be the longest element of the Weyl group $W$ of $(\g,\,\h)$ and $w_0'$ the longest element of the Weyl group $W_{\pi'}$.

For all $\alpha\in\pi$, we set ${\bf j}(\alpha)=-w_0(\alpha)$ and for all $\alpha\in\pi'$, we set ${\bf i}(\alpha)=-w_0'(\alpha)$. For $\alpha\in\pi\setminus\pi'$,  let $r\in\mathbb N$ be the smallest integer such that ${\bf j}({\bf i j})^r(\alpha)\not\in\pi'$. We set ${\bf i}(\alpha)={\bf j}({\bf i j})^r(\alpha)$.

By \cite[3.2]{FJ1} we have that ${\rm GKdim }\,Y(\p)=\ind{\p}=\lvert E(\pi')\rvert$.

Denote by $\{\varpi_{\alpha}\}_{\alpha\in\pi}$ (resp. $\{\varpi'_{\alpha}\}_{\alpha\in\pi'}$) the set of fundamental weights associated to $\pi$ (resp. to $\pi'$); the same sets sometimes are denoted by $\{\varpi_i\}_{\alpha_i\in\pi}$ and $\{\varpi'_i\}_{\alpha_i\in\pi'}$ respectively. Let $\mathcal B_{\pi}$ (resp. $\mathcal B_{\pi'}$) be the set of weights of the Poisson semicentre of $S(\n\oplus\h)$ (resp. $S(\n_{\pi'}\oplus\h')$): the weights of the generators of the Poisson semicentre of a Borel are listed in~\cite[Tables I and II]{J1} and \cite[Table]{FJ2} for an erratum.

For all $\Gamma\in E(\pi')$, set 
$\delta_{\Gamma}=-\sum_{\gamma\in\Gamma}\varpi_{\gamma}-\sum_{\gamma\in {\bf j}(\Gamma)}\varpi_{\gamma}+\sum_{\gamma\in\Gamma\cap\pi'}\varpi'_{\gamma}+\sum_{\gamma\in {\bf i}(\Gamma\cap\pi')}\varpi'_{\gamma}$

and $\varepsilon_{\Gamma}=\begin{cases}
1/2&{\rm if}\; \Gamma={\bf j}(\Gamma), \;{\rm and}\; \sum_{\gamma\in\Gamma}\varpi_{\gamma}\in\mathcal B_{\pi}, \;{\rm and}\;\sum_{\gamma\in\Gamma\cap\pi'}\varpi'_{\gamma}\in\mathcal B_{\pi'}.\\
 1& \;{\rm otherwise.}\\
\end{cases}$
It is shown in~\cite[Thm. 6.7]{J7} that

$${\rm ch}\,\mathcal A=\prod_{\Gamma\in E(\pi')}(1-e^{\delta_{\Gamma}})^{-1}\le \ch Y(\p)\le\prod_{\Gamma\in E(\pi')}(1-e^{\varepsilon_{\Gamma}\delta_{\Gamma}})^{-1}={\rm ch}\,\mathcal B.$$

In particular, if for all $\Gamma\in E(\pi')$, $\varepsilon_{\Gamma}=1$, the above inequalities are equalities and $Y(\p)$ is a polynomial algebra over $k$ by~\cite{FJ2}.

An adapted pair for $\p$ is a pair $(h,\,y)\in\h_{\Lambda}\times\p^*$ such that $y$ is regular in $\p^*$, and $(\ad h)\,y=-y$ where $\ad$ denotes the coadjoint action of $\p$ on $\p^*$.

Assume that there exists an adapted pair $(h,\,y)\in\h_{\Lambda}\times\p^*$ for $\p$. One may choose subsets $S,\,T\subset \Delta^+\sqcup\Delta^-_{\pi'}$ such that $y=\sum_{\gamma\in S}a_{\gamma}x_\gamma$, with $a_\gamma\in k\setminus \{0\}$ for all $\gamma\in S$, and $\p^*=(\ad\p)\,y\oplus\g_T$. Note that we may choose $T$ such that $\lvert T\rvert=\ind\p$. Assume further that $S_{\mid\h_{\Lambda}}$ is a basis for $\h_{\Lambda}^*$.
Then for each $\gamma\in T$ there exists a unique $t(\gamma)\in\mathbb Q S$ such that $\gamma+t(\gamma)$ vanishes on $\h_{\Lambda}$.

By~\cite[Lem. 6.11]{J6bis}
$$\ch Y(\p)\le\prod_{\gamma\in T} (1-e^{-(\gamma+t(\gamma))})^{-1}=\mathcal B'$$
and we will call the right hand side $\mathcal B'$ an ``improved upper bound'' for $\ch Y(\p)$; in this work it is indeed always an improvement of the upper bound ${\rm ch}\,\mathcal B$ mentioned above.

Moreover by~\cite[Lem. 6.11]{J6bis} if the above lower bound  ${\rm ch}\,\mathcal A$ and this improved upper bound $\mathcal B'$ coincide then the restriction map gives an isomorphism of algebras $Y(\p)\simeq R[y+\g_T]$, where $R[y+\g_T]$ is the ring of polynomial functions on $y+\g_T$, isomorphic to
$S(\g_T^*)$. Hence $Y(\p)$ is a polynomial algebra over $k$ and $y+\g_T$ is an algebraic slice in the sense of~\cite[7.6]{J8}, also called a Weierstrass section in~\cite{FJ4} and by~\cite{FJ4} it is also an affine slice in the sense of~\cite[7.3]{J8} for the coadjoint action of the adjoint group of $\p$ on $\p^*$.

Assume that there exists an adapted pair $(h,\,y)$ for $\p$ and denote by $V$ an $h$-stable complement of $(ad\,\p)\,y$ in $\p^*$.
Assume further that $Y(\p)$ is a polynomial algebra and let $f_1,\ldots,\,f_l$ be homogeneous generators for $Y(\p)$ ($l=\ind\p$). Then by~\cite[Cor. 2.3]{JS}
if $m_1,\ldots,\,m_l$ are the eigenvalues of $h$ on an $h$-stable basis of $V$, one has that $\deg\,f_i=m_i+1$ for all $1\le i\le l$, up to a permutation of indices.

\section{A lemma of regularity.}\label{regularity}

Keep the notations of the previous section and let $f\in\g$ and $\Phi_f:\;\g\times\g\longrightarrow k$ be the skew-symmetric bilinear form on $\g$ defined by $\Phi_f(x,\,x')=K(f,\,[x,\,x'])$.
Here we recall (\cite[Def. 2]{FL}) the definition of a Heisenberg set, of centre $\gamma\in\Delta$. It is a subset $\Gamma_{\gamma}$ of $\Delta$ such that $\gamma\in\Gamma_{\gamma}$ and for all $\alpha\in\Gamma_{\gamma}\setminus\{\gamma\}$, there exists a (unique) $\alpha'\in\Gamma_{\gamma}\setminus\{\gamma\}$ such that $\alpha+\alpha'=\gamma$.

\begin{Example}\label{MaxHeisenbergset}\rm

Set $\Delta=\sqcup\Delta_i$ where $\Delta_i$ is an irreducible root system and let $\beta_i$ be the unique highest root of $\Delta_i$.

Take $(\Delta_i)_{\beta_i}:=\{\alpha\in \Delta_i\,|\, (\alpha, \beta_i)=0\}$ and decompose it into irreducible root systems $\Delta_{ij}$ with highest roots $\beta_{ij}$.

Continuing we obtain a set of strongly orthogonal positive roots $\beta_K$, called the {\it Kostant cascade}, indexed by elements $K\in \mathbb N\cup \mathbb N^2\cup\cdots.$

The sets $H_{\beta_K}:=\{\alpha\in \Delta_K\,|\, (\alpha, \beta_K)>0\}$ are Heisenberg sets of centre $\beta_K$.
They are maximal Heisenberg sets, among the Heisenberg sets which are included in $\Delta^+$.

\end{Example}

Let $\p_{\pi',\,\Lambda}=\n\oplus\h_{\Lambda}\oplus\n^-_{\pi'}\simeq(\p_{\pi',\,\Lambda}^-)^*$ be the truncated parabolic subalgebra of $\g$ associated to $\pi'\subset \pi$. Let $S$ be a subset of $\Delta^+\sqcup\Delta^-_{\pi'}$ and for all $\gamma\in S$ choose a Heisenberg set $\Gamma_{\gamma}$ of centre $\gamma$ in $\Delta^+\sqcup\Delta^-_{\pi'}$. Assume that the sets $\Gamma_{\gamma}$ are disjoint and set $\Gamma=\bigsqcup_{\gamma\in S}\Gamma_{\gamma}$ and
$y=\sum_{\gamma\in S}a_{\gamma}x_{\gamma}\in\p_{\pi',\,\Lambda}$, where $a_{\gamma}\in k\setminus\{0\}$ for all $\gamma\in S$. Set $O=\bigsqcup_{\gamma\in S}\Gamma_{\gamma}^0$, with $\Gamma_{\gamma}^0=\Gamma_{\gamma}\setminus\{\gamma\}$, and $\o=\g_{-O}$.

The lemma below follows exactly like~\cite[Lem. 6]{FL}.\smallskip

\begin{lemma}\label{refineregular} Retain the above notations and hypotheses and assume further that
\begin{enumerate}
\item[(i)] The restriction of $\Phi_y$ to $\o\times\o$ is non-degenerate.
\item[(ii)] $S_{\mid\h_{\Lambda}}$ is a basis for $\h_{\Lambda}^*$.
\item[(iii)] $\lvert T\rvert=\ind\p_{\pi',\,\Lambda}$, where $T=(\Delta^+\sqcup\Delta^-_{\pi'})\setminus\Gamma$.
\end{enumerate}
Then $\p_{\pi',\,\Lambda}=(\ad \p_{\pi',\,\Lambda}^-)\,y\oplus\g_T$, where $\ad$ denotes the coadjoint action. In particular, $y$ is regular in $\p_{\pi',\,\Lambda}$. Moreover, if we uniquely define $h\in \h_{\Lambda}$ by the relations $\gamma(h)=-1$ for all $\gamma\in S$, then $(h,\,y)$ is an adapted pair for $\p_{\pi',\,\Lambda}^-$.
\end{lemma}

\begin{Remark}\rm
In \cite{FL} adapted pairs for maximal parabolic subalgebras when both bounds ${\rm ch}\,\mathcal A$ and ${\rm ch}\,\mathcal B$ coincide were essentially obtained by taking part of the Kostant cascade for the set $S$  and for all $\gamma\in S\cap \Delta^+$, resp. $\gamma\in S\cap \Delta^-_{\pi'}$, the corresponding maximal Heisenberg set (\ref{MaxHeisenbergset}) in $\Delta^+$, resp. in $\Delta^+_{\pi'}$, for the Heisenberg set $\Gamma_{\gamma}$, resp. $-\Gamma_{-\gamma}$. Unfortunately  in the case of maximal parabolic subalgebras when the bounds ${\rm ch}\,\mathcal A$ and ${\rm ch}\,\mathcal B$ do not coincide, the restriction to $\h_{\Lambda}$ of the Kostant cascade does no more give a basis for $\h_{\Lambda}^*$.

\end{Remark}

\section{Stationary  roots.}\label{sr}
Keep the notations of Sections~\ref{standardnotation} and~\ref{regularity}.
Given $\gamma\in S$, for all $\alpha\in\Gamma_{\gamma}^0$ denote by $\alpha'$ the unique root in $\Gamma_{\gamma}^0$ such that $\alpha+\alpha'=\gamma$ and let $\theta_{\gamma}$ be the involution in $\Gamma_{\gamma}^0$ mapping $\alpha\in\Gamma_{\gamma}^0$ to $\alpha'$. Denote by $\theta$ the involution in $O$ induced by all $\theta_{\gamma},\, \gamma\in S$.

Clearly, the non-degeneracy of the restriction of $\Phi_y$ to $\o\times\o$ is immediate if, for all $\alpha\in O$, the only root $\beta$ in $O$ such that $\alpha+\beta\in S$ is $\beta=\theta(\alpha)$.

Unfortunately this will not be the case in general but  Lemma~\ref{non-degeneracy} below will give sufficient conditions for the non-degeneracy of the restriction of $\Phi_y$ to $\o\times\o$. To state this lemma, we need further notations. In particular
 for each root $\alpha\in O$, we set $S_{\alpha}=\{\beta\in O\mid\alpha+\beta\in S\}$ and for all $n\ge 1,\, O_n=\{\alpha\in O\,\mid\, \lvert S_{\alpha}\rvert =n\}$. Note that $O_1=\{\alpha\in O\mid \forall\, \beta\in O,\,\alpha+\beta\in S\Longrightarrow\beta=\theta(\alpha)\}$.

 Let $\alpha\in O$. Set $\alpha^0=\alpha$ and for all $i\in\mathbb N$ define $\alpha^{i}\in O$ inductively as follows. If $\theta(\alpha^{i})\in O_1$, set $\alpha^{i+1}=\alpha^{i}$. Otherwise, let $\alpha^{i+1}\neq \alpha^{i}$ be a root in $O$ such that $\alpha^{i+1}+\theta(\alpha^{i})\in S$. For all $i\in \mathbb N$, set $\alpha^{(i)}=\theta(\alpha)^i$. Note that, if $\alpha\in O_2$, then $\alpha^{(1)}$ is the only root in $O$ distinct from $\theta(\alpha)$ such that $\alpha^{(1)}+\alpha\in S$. Similarly if $\theta(\alpha)\in O_2$, then $\alpha^1$ is the only root in $O$ distinct from $\alpha$ such that $\alpha^1+\theta(\alpha)\in S$. Observe that $\theta(\alpha)^{(i)}=\alpha^i$.

We will say that $(\alpha^i)_{i\in\mathbb N}$ is {\it a} sequence of roots in $O$ constructed from $\alpha$; such a sequence always exists but in general is not unique. If for all $i\in\mathbb N$, $\theta(\alpha^i)\in O_1\sqcup O_2$, then $(\alpha^i)_{i\in\mathbb N}$ will be called {\it the} sequence of roots in $O$ constructed from $\alpha$, since in this case, $\alpha^i$ is uniquely defined, for all $i\in\mathbb N$. Note that if $\theta(\alpha^i)\in O_1$ for some $i\in \mathbb N$, then $\alpha^j=\alpha^i$ for all $j\ge i$. Conversely, if $\alpha^i=\alpha^{i+1}$, then $\theta(\alpha^i)\in O_1$ and $\alpha^j=\alpha^i$, for all $j\ge i$. We call a minimal such $i$ the rank of the sequence $(\alpha^j)_{j\in\mathbb N}$ and we say that the sequence is {\it stationary at rank $i$}. Note that if $\theta(\alpha^i)\not\in O_1$ then $\alpha^{i+1}\not\in O_1$.

Let $\alpha\in O$ and
set $A_{\alpha}=\{\alpha^i,\,\theta(\alpha^i)\mid i\in\mathbb N\}$ for a sequence $(\alpha^i)_{i\in\mathbb N}$ of roots in $O$ constructed from $\alpha$.

\begin{Remark} \rm \label{rem1}

Let $\alpha\in O$ and assume  that $A_{\alpha}\cup A_{\theta(\alpha)}\subset O_1\sqcup O_2$.

Let $i,\ j\in\mathbb N$.

(1) One has that $(\alpha^i)^j=\alpha^{i+j}$ and $(\alpha^{(i)})^j=\alpha^{(i+j)}$.\par

(2) Assume that $i\ge 1$ and that $\theta(\alpha^{i-1})\in O_2$. If $j\le i$, then $(\alpha^i)^{(j)}=\theta(\alpha^{i-j})$ and if $j\ge i$, then $(\alpha^i)^{(j)}=\alpha^{(j-i)}$.\par

(3) Assume that $i\ge 1$ and that $\theta(\alpha^{(i-1)})\in O_2$. If $j\le i$, then $(\alpha^{(i)})^{(j)}=\theta(\alpha^{(i-j)})$ and if $j\ge i$, then $(\alpha^{(i)})^{(j)}=\alpha^{j-i}$.\par

\end{Remark}

\begin{proof}
The definition of the roots $\alpha^i$ and $\alpha^{(i)}$ and an induction 
on $j$, noting that $\alpha^{i+1}=(\alpha^i)^1$ and that $\alpha^{(i+1)}=(\alpha^{(i)})^1$, give the assertions.
\end{proof}

\begin{Definition}\label{defsr}\rm
Let $\alpha\in O$.
We say that $\alpha$ is a {\it stationary root} if $A_{\alpha}\cup A_{\theta(\alpha)}\subset O_1\sqcup O_2$ and if the sequences $(\alpha^i)_{i\in\mathbb N}$ and $(\alpha^{(i)})_{i\in\mathbb N}$ are stationary.

The set of stationary roots in $O$ will be denoted by $O_{st}$.
\end{Definition}

\begin{Remark}\label{rem 2} \rm

 Let $\alpha\in O$.
 If $\alpha\in O_{st}$ then $A_{\alpha}\cup A_{\theta(\alpha)}\subset O_{st}$ and conversely if  $A_{\alpha}\cup A_{\theta(\alpha)}\subset O_1\sqcup O_2$ and if there exists $i_0\in\mathbb N$ such that $\alpha^{i_0}$ or $\alpha^{(i_0)}$ belongs to $O_{st}$, then $\alpha\in O_{st}$. 
\end{Remark}

\begin{proof}
This easily follows from (1), (2) and (3) of Remark~\ref{rem1}.
\end{proof}

\begin{lemma}\label{stat}
Let $\alpha\in O_{st}$. Let $\vartheta:O\to O$ be a permutation such that, for all $\gamma\in O$, $\gamma+\vartheta(\gamma)\in S$. Then the restriction of $\vartheta$ to $A_{\alpha}\cup A_{\theta(\alpha)}$ coincides with the involution $\theta$.
\end{lemma}

\begin{proof}
Denote by $n_0$ (resp. $n_1$) the rank of the stationary sequence $(\alpha^i)_{i\in\mathbb N}$ (resp. $(\alpha^{(i)})_{i\in\mathbb N}$).

Since $\theta(\alpha^{n_0})\in O_1$ (resp. $\theta(\alpha^{(n_1)})\in O_1$) the map $\vartheta$ necessarily sends $\theta(\alpha^{n_0})$ (resp. $\theta(\alpha^{(n_1)})$) to $\alpha^{n_0}$ (resp. $\alpha^{(n_1)}$).

Then a decreasing induction on $i$ gives that, for all $0\le i\le n_0$,  we have $\vartheta(\theta(\alpha^i))=\alpha^i$.
Similarly we obtain that, for all $0\le i\le n_1$, $\vartheta(\theta(\alpha^{(i)}))=\alpha^{(i)}$.
An increasing induction on $i$ proves then that $\vartheta(\alpha^i)=\theta(\alpha^i)$ for all $0\le i\le n_0$
and that $\vartheta(\alpha^{(i)})=\theta(\alpha^{(i)})$ for all $0\le i\le n_1$. 
\end{proof}

\section{Cyclic roots}\label{Seccyc}

We also need to define what we have called a {\it cyclic root}. Recall the notations and hypotheses of Section~\ref{sr}.

\begin{Definition}\label{cyc}\rm

Let $\alpha\in O$. We say that $\alpha$ is {\it a cyclic root}  if there exist $\beta,\,\gamma\in O$ such that the following conditions are satisfied:
\begin{enumerate}
\item[(i)] $\theta(\alpha)+\gamma=\beta+\theta(\beta)$
\item[(ii)] $\theta(\gamma)+\beta=\alpha+\theta(\alpha)$
\item[(iii)] $\theta(\beta)+\alpha=\gamma+\theta(\gamma)$
\item[(iv)] $\{\alpha,\,\beta,\,\gamma,\,\theta(\alpha),\,\theta(\beta),\,\theta(\gamma)\}\subset O_2\sqcup O_3$
\item[(v)] $\lvert\{\alpha,\,\beta,\,\gamma,\,\theta(\alpha),\,\theta(\beta),\,\theta(\gamma)\}\rvert=6$
\item[(vi)] If $\delta\in\{\alpha,\,\beta,\,\gamma,\,\theta(\alpha),\,\theta(\beta),\,\theta(\gamma)\}\cap O_3$, then there exists  $\tilde \delta\in S_{\delta}$  such that $\tilde\delta\in O_2$ and $\theta(\tilde\delta)\in O_1$.
\end{enumerate}

The set of cyclic roots in $O$ is denoted by $O_{cyc}$.
\end{Definition}

\noindent For $\alpha\in O_{cyc}$, set $C_{\alpha}=\{\alpha,\,\beta,\,\gamma,\,\theta(\alpha),\,\theta(\beta),\,\theta(\gamma)\}$. Note that
for $\delta\in C_{\alpha}\cap O_3$ then, with the above notations, $\tilde\delta$ is unique and $S_\delta\setminus S_\delta\cap C_\alpha=\{\tilde\delta\}$.

\smallskip

\begin{Remark}\label{remcyc} \rm (1) If $\alpha\in O_{cyc}$ then all roots in $C_{\alpha}$ are cyclic roots.\par
(2) Suppose that $\alpha\in O_{cyc}$ and that $C_{\alpha}\subset O_2$. Then $\alpha\not\in O_{st}$. Indeed the cyclic relations (i), (ii) and (iii) imply that $\alpha^1=\gamma$, $\alpha^2=\beta$ and $\alpha^3=\alpha$, hence the sequence $(\alpha^i)_{i\in\mathbb N}$ is not stationary. \par
(3) Only two of conditions (i), (ii), (iii) above are necessary. Indeed any two of them imply the third one.\par
(4) With the above notations, let $\alpha\in O$ such that $\alpha=\tilde\beta$, with $\beta\in O_{cyc}\cap O_3$. Then $\alpha\not\in O_{st}$. Indeed since $\theta(\tilde\beta)\in O_1$, we have that $\alpha^1=\alpha$, but
since $\tilde\beta\in S_{\beta}$, one has that $\alpha^{(1)}=\beta$.

\end{Remark}
\smallskip

\begin{lemma}\label{permcyc}
Let $\vartheta:O\to O$ be a permutation such that for all $\gamma\in O$, $\gamma+\vartheta(\gamma)\in S$. Then $\vartheta$ exchanges $\tilde\delta$ and $\theta(\tilde\delta)$, where $\tilde\delta$ is the unique root in $S_{\delta}$ given by condition (vi) of Definition~\ref{cyc}, for any $\delta\in O_{cyc}\cap O_3$.
\end{lemma}

\begin{proof}

Denote by $\{\delta_1,\,\delta_2,\,\delta_3\}$ the set of cyclic roots $\{\alpha,\,\beta,\gamma\}$ in $O$ which satisfy the cyclic relations (i), (ii) and (iii) of Definition~\ref{cyc}. Assume that $\delta=\delta_1\in O_3$ and denote by $\tilde\delta=\tilde\delta_1$ the root in $O$ satisfying condition (vi) of Definition~\ref{cyc}.

Since $\theta(\tilde\delta)\in O_1$, we have necessarily that $\vartheta(\theta(\tilde\delta))=\tilde\delta$. Now since $\tilde\delta\in O_2$, we have that $\vartheta(\tilde\delta)=\theta(\tilde\delta)$ or $\delta$, since moreover $\tilde\delta\in S_{\delta}$.

Assume that $\vartheta(\tilde\delta)=\delta$. Then necessarily $\vartheta(\theta(\delta))=\delta_3$ by condition (i) of Definition~\ref{cyc}. Then $\vartheta(\theta(\delta_3))=\delta_2$ by condition (ii) of Definition~\ref{cyc}. But condition (iii) of Definition~\ref{cyc} implies then that $\vartheta(\theta(\delta_2))=\delta_1=\delta$ which is not possible, since $\delta$ has already a preimage by $\vartheta$.
\end{proof}

\section{A lemma of non-degeneracy.}

Let $S^+$ (resp. $S^-$) be the subset of $S$ containing those $\gamma\in S$ for which $\Gamma_\gamma\subset \Delta^+$ (resp. $\Gamma_\gamma\subset \Delta^-_{\pi'}$).

Let $S^m$ be the subset of $S$ containing those $\gamma\in S$, for which the Heisenberg set $\Gamma_{\gamma}$ contains {\bf both} positive and negative roots in $\Delta^+\sqcup\Delta^-_{\pi'}$.

We have $S=S^+\sqcup S^-\sqcup S^m$ and we set
$\Gamma^\pm=\bigsqcup_{\gamma\in S^\pm}\Gamma_{\gamma},\, \Gamma^m=\bigsqcup_{\gamma\in S^m}\Gamma_{\gamma}$; then $\Gamma=\Gamma^+\sqcup\Gamma^-\sqcup\Gamma^m$.

For all $\gamma\in S$, recall that $\Gamma_{\gamma}^0=\Gamma_{\gamma}\setminus\{\gamma\}$,
and set $O^\pm=\bigsqcup_{\gamma\in S^\pm}\Gamma_{\gamma}^0,\,O^m=\bigsqcup_{\gamma\in S^m}\Gamma_{\gamma}^0$; we have $O=O^+\sqcup O^-\sqcup O^m$.

Set also $\o^\pm=\g_{-O^\pm}$ and $\o^m=\g_{-O^m}$ so that $\o=\g_{-O}=\o^+\oplus\o^-\oplus\o^m$.

\begin{lemma}\label{non-degeneracy}

Assume  that:
\begin{enumerate}
\item $S_{\mid\h_{\Lambda}}$ is a basis for $\h_{\Lambda}^*$.
\item If $\alpha\in \Gamma_{\gamma}^0$, with $\gamma\in\ S^+$, then $S_{\alpha}\cap O^+=\{\theta(\alpha)\}$.

\item  If $\alpha\in \Gamma_{\gamma}^0$, with $\gamma\in\ S^-$, then $S_{\alpha}\cap O^-=\{\theta(\alpha)\}$.

 \item If $\alpha\in O$, with $S_{\alpha}\cap O^m\neq\emptyset$, then $\alpha\in O_{st}$ or $\alpha\in O_{cyc}$ or there exists $\beta\in O_{cyc}\cap O_3$ such that $\alpha=\tilde\beta$ or $\theta(\alpha)=\tilde\beta$, where $\tilde\beta$ is the unique root in $S_{\beta}$ such that $\tilde\beta\in O_2$ and $\theta(\tilde\beta)\in O_1$.
 \end{enumerate}
Then the restriction of the bilinear form $\Phi_y$ to $\o\times\o$ is non-degenerate.
\end{lemma}

\begin{proof}

Let $\rho$ be the linear form on $\h^*$ defined by $\rho(\alpha)=1$ for all $\alpha\in\pi$
and define $z(t)=\sum_{\gamma\in S}t^{\lvert\rho(\gamma)\rvert}a_{\gamma}x_{\gamma}$ for $t\in k$.

Set $d(t)=\det({\Phi_{z(t)}}_{\mid\o\times\o})$ which is a polynomial in $t$ and let $H_{\Lambda}$ denote the adjoint group of $\h_{\Lambda}$.

Since by hypothesis (1) the elements of $S_{\mid\h_{\Lambda}}$ are linearly independent it follows that $z(ct_0)$ and $z(t_0)$ are in the same $H_{\Lambda}$-coadjoint orbit for $t_0\in k$ and for all $c\in k\setminus\{0\}$. Moreover $\o\times\o$ is stable under the adjoint action of $H_{\Lambda}$.
Then the degeneracy of the restriction of the bilinear form $\Phi_{z(t_0)}$ on $\o\times\o$ is equivalent to the degeneracy of the restriction of the bilinear form $\Phi_{z(ct_0)}$ on $\o\times\o$ for all $c\in k\setminus\{0\}$, that is, $d(t_0)=0$ is equivalent to $d(ct_0)=0$ for all $c\in k\setminus\{0\}$. It follows that either $d(t)$ is identically zero or it annihilates only at $t=0$. Hence $d(t)$ is a multiple of a single power of $t$ (see also \cite[Rem. 8.4]{J5}).

Let $\alpha\in O$ be such that $S_{\alpha}\cap O^m\neq\emptyset$.

Assume first that $\alpha\in O_{st}$. Then by lemma~\ref{stat}, the only factor involving $\alpha$  in $\det(\Phi_{z(t)_{\mid\o\times\o}})$ is $t^{2\lvert\rho(\alpha+\theta(\alpha))\rvert}$.

Assume now that $\alpha=\tilde\beta$, or $\theta(\alpha)=\tilde\beta$ with $\tilde\beta\in S_{\beta}$, $\beta\in O_{cyc}\cap O_3$, satisfying condition (vi) of Definition~\ref{cyc}. Then, by Lemma~\ref{permcyc}, in the expansion of the determinant of ${\Phi_{z(t)}}_{\mid\o\times \o}$ the only factor involving $\alpha$  is $t^{2\lvert\rho(\alpha+\theta(\alpha))\rvert}$.

Assume  that $\alpha\in O_{cyc}$ and consider $C_{\alpha}=\{\alpha,\,\beta,\,\gamma,\,\theta(\alpha),\,\theta(\beta),\,\theta(\gamma)\}$ verifying conditions (i)-(vi) of Definition~\ref{cyc}.
Then the matrix of $\Phi_{z(t)_{\mid{\g_{-C_{\alpha}}\times\g_{-C_{\alpha}}}}}$ is, up to a nonzero scalar, of the form

$$\begin{pmatrix}
0&0&0&t^{\lvert\rho(s_1)\rvert}&t^{\lvert\rho(s_3)\rvert}&0\\
0&0&0&0&t^{\lvert\rho(s_2)\rvert}&t^{\lvert\rho(s_1)\rvert}\\
0&0&0&t^{\lvert\rho(s_2)\rvert}&0&t^{\lvert\rho(s_3)\rvert}\\
-t^{\lvert\rho(s_1)\rvert}&0&-t^{\lvert\rho(s_2)\rvert}&0&0&0\\
-t^{\lvert\rho(s_3)\rvert}&-t^{\lvert\rho(s_2)\rvert}&0&0&0&0\\
0&-t^{\lvert\rho(s_1)\rvert}&-t^{\lvert\rho(s_3)\rvert}&0&0&0
\end{pmatrix}$$
where $s_1=\alpha+\theta(\alpha)$, $s_2=\beta+\theta(\beta)$ and $s_3=\gamma+\theta(\gamma)$.

Hence up to a nonzero scalar, $$\det(\Phi_{z(t)_{\mid{\g_{-C_{\alpha}}\times\g_{-C_{\alpha}}}}})=t^{2(\lvert\rho(s_1)\rvert+\lvert\rho(s_2)\rvert+\lvert\rho(s_3)\rvert)}=t^{2(\lvert\rho(\alpha+\theta(\alpha))\rvert+\lvert\rho(\beta+\theta(\beta))\rvert+\lvert\rho(\gamma+\theta(\gamma))\rvert)}$$
and by Lemma~\ref{permcyc}, it follows that the only factor involving $\alpha$  in $\det(\Phi_{z(t)_{\mid\o\times\o}})$ is $t^{2\lvert\rho(\alpha+\theta(\alpha))\rvert}$.

Let now $\alpha\in O^{\pm}$ such that $S_{\alpha}\cap O^{\mp}\neq\emptyset$ and $\beta\in S_{\alpha}\cap O^{\mp}$. By the above, if there is a factor in $\det(\Phi_{z(t)_{\mid\o\times\o}})$ which involves $\alpha$ and $\beta$, that is, if $t^{\lvert\rho(\alpha+\beta)\rvert}$ appears as a factor in $\det(\Phi_{z(t)_{\mid\o\times\o}})$, then necessarily
$S_{\alpha}\cap O^m=\emptyset$ and $S_{\theta(\alpha)}\cap O^m=\emptyset$.

Then observe that $\lvert\rho(\alpha+\beta)\rvert<\lvert\rho(\alpha)\rvert+\lvert\rho(\beta)\rvert,$ \par\noindent
${\rm whilst }\;\lvert\rho(\alpha+\theta(\alpha))\rvert=\lvert\rho(\alpha)\rvert+\lvert\rho(\theta(\alpha))\rvert\;{\rm and}\;\;\lvert\rho(\beta+\theta(\beta))\rvert=\lvert\rho(\beta)\rvert+\lvert\rho(\theta(\beta))\rvert.$

Since $d(t)$ is a multiple of a single power of $t$, the above observations and conditions (2) and (3) imply that $t^{\lvert\rho(\alpha+\beta)\rvert}$ cannot appear as a factor in $\det(\Phi_{z(t)_{\mid\o\times\o}})$. (See also the proof of~\cite[Lemma 8.5]{J5}).

Denote by $\widetilde {O}$ a choice of representatives in $O$ modulo the involution $\theta$. Then, up to a nonzero scalar,

$$\begin{array}{rll}
d(t)=\det(\Phi_{z(t)\mid\o\times\o})=
\displaystyle\prod_{\alpha\in\widetilde{ O}}t^{2\lvert\rho(\alpha+\theta(\alpha))\rvert}\\
\end{array}$$

Thus $\det(\Phi_{z(t)\mid\o\times\o})\neq 0$ for $t\neq 0$ and the assertion of the lemma follows.
\end{proof}

\begin{Remark} \rm If $S^m=\emptyset$ then condition (4) is empty and the above lemma is~\cite[Lemma 5]{FL}. \end{Remark}

By Lemmata~\ref{refineregular} and \ref{non-degeneracy} we obtain the following corollary.

\begin{corollary}\label{conc}
Assume that the hypotheses of the previous lemma hold and that $\lvert T\rvert=\ind\p_{\pi',\, \Lambda}$, where $T=(\Delta^+\sqcup\Delta^-_{\pi'})\setminus\Gamma$. Recall that $y=\sum_{\gamma\in S}a_{\gamma}x_{\gamma}$ and define $h\in\h_{\Lambda}$ by $\gamma(h)=-1$ for all $\gamma\in S$. Then $(h,\,y)$ is an adapted pair for $\p^-_{\pi',\,\Lambda}$.
\end{corollary}

In what follows, we construct adapted pairs for the truncated maximal parabolic subalgebras $\p$ in type ${\rm B}$ where the lower and upper bounds ${\rm ch}\,\mathcal A$ and ${\rm ch}\,\mathcal B$ of Section~\ref{standardnotation} do not coincide; $\p$ is associated to the subsystem $\pi'$ of $\pi$ obtained by suppressing a root of even index. The construction of an adapted pair in these cases is much more involved than in~\cite{FL}.

\section{Type B}\label{SecB}
In this section, $\g$ is a simple Lie algebra of type ${\rm B}_n$ ($n\ge 2$) and $\p=\p^-_{\pi',\,\Lambda}$ is the truncated maximal parabolic subalgebra associated to the subset $\pi'=\pi\setminus\{\alpha_s\}$ of $\pi$ with $s$ even ($2\le s\le n$). In this case the lower and the upper bounds for $\ch Y(\p)$, ${\rm ch}\,\mathcal A$ and ${\rm ch}\,\mathcal B$ of Section~\ref{standardnotation}, do not coincide~\cite[4.1]{FL} (except when $n=s=2$ or $n=s=4$, cases that we will however also consider in the following).

We will construct an adapted pair $(h,\,y)$ for $\p$, a slice for its coadjoint action and show that $Y(\p)$ is polynomial in $\ind\,\p$ generators. It will follow by the discussion in the introduction that the field $C(\p^-_{\pi'})$ of invariant fractions is a purely transcendental extension of $k$.

As we said above, it is enough to find sets $S,\,T$ that satisfy the conditions of Lemma \ref{non-degeneracy} and of Corollary~\ref{conc}.

Recall that the truncated Cartan subalgebra of $\p$ is the Cartan subalgebra of the Levi factor, namely $\h_{\Lambda}=\h'=\bigoplus_{1\le i\le n, i\neq s} k\,{\alpha_i}^{\vee}$.

Denote by $\{\varepsilon_i\,|\,1\le i\le n\}$ an orthonormal basis of $\mathbb R^n$ according to which
the simple roots $\alpha_i$ ($1\le i\le n$) of $\g$ are expanded as in~\cite[Planche II]{BOU}.
Then the Kostant cascade formed by the strongly orthogonal positive roots  $\beta_i$ and $\beta_{i'}$ is given in~\cite[Table I]{FL} or in~\cite[Table II]{J1}. Recall that $\beta_i=\varepsilon_{2i-1}+\varepsilon_{2i}$ for all $1\le i\le[n/2]$,  and for $n$ odd, $\beta_{(n+1)/2}=\varepsilon_n=\alpha_n$ and $\beta_{i'}=\alpha_{2i'-1}=\ep_{2i'-1}-\ep_{2i'}$ for all $1\le i'\le [n/2]$.
\smallskip

Set $S=S^+\sqcup S^-\sqcup S^m$ with the following subsets $S^\pm$ and $S^m$:
\begin{itemize}

\item If $n=s$, $S^+=\{\beta_i\mid 1\le i\le s/2-1\}$.
 \item If $n>s$, $S^+=\{\beta_i,\,\varepsilon_{s-1}+\varepsilon_{s+1},\,\varepsilon_{2j}+\varepsilon_{2j+1}\,\mid 1\le i\le s/2-1,\,s/2+1\le j\le [(n-1)/2]\}.$ 
 \item
 $S^-=\displaystyle\{\varepsilon_{s-i}-\varepsilon_i,\,-\beta_j=-\varepsilon_{2j-1}-\varepsilon_{2j}\,\mid 1\le i\le s/2-1,\,s/2+1\le j\le [n/2] \}.$
 \item
$S^m=\{\ep_s\}$. 
 \end{itemize}

Clearly, $S\subset\Delta^+\sqcup\Delta^-_{\pi'}$ and $\lvert S\rvert=n-1=\dim\h_{\Lambda}$. We first show below that condition $(1)$ of Lemma~\ref{non-degeneracy} holds.

\begin{lemma}
$S_{\mid\h_{\Lambda}}$ is a basis for $\h_{\Lambda}^*$.
\end{lemma}

\begin{proof}
It is sufficient to show that if $S=\{s_1,\ldots, s_{n-1}\}$ and $\{h_1,\ldots, h_{n-1}\}$ is a basis of $\h_{\Lambda}$, then $\det(s_i(h_j))_{i,j}\neq 0$. We will prove this statement by induction on $n$. We choose $\{\alpha_i^\vee\,|\,1\le i\le n,\, i\ne s\}$ as a basis of $\h_{\Lambda}$.

Add temporarily a lower subscript $n$ to $S^{\pm},\, \pi,\, \h'=\h_\Lambda$ to emphasize that they are defined for type ${\rm B}_n$ and observe that $S^m$ does not depend on $n$.

Identify an element $(x_1,x_2,\ldots,x_n)\in\mathbb R^n$ with the element $(x_1, x_2, \ldots, x_n,0)\in\mathbb R^{n+1}$. Observe that $S^+_{s+1}=S^+_{s}\sqcup\{\varepsilon_{s-1}+\varepsilon_{s+1}\}$, whereas for $n$ even and $n\ge s+2$ we have $S^+_{n+1}=S^+_{n}\sqcup\{\varepsilon_n+\varepsilon_{n+1}\}$ and for $n$ odd we have $S^+_{n+1}=S^+_{n}$.

Similarly for $n$ even, $S_{n+1}^-=S_n^-$ and for $n$ odd, $S^-_{n+1}=S^-_n\sqcup\{-\varepsilon_n-\varepsilon_{n+1}\}$. Finally set $S_n=S_n^+\sqcup S_n^-\sqcup S^m$.

We first consider the case $n=s$.

If $n=s=2$ then $S=\{s_1=\varepsilon_2=\alpha_2\}$ and  $\det\,(s_1(\alpha_1^\vee))=-1\neq 0$.

Assume now that $n\ge 4$ and $n=s$. Then $S=\{\varepsilon_n,\,\varepsilon_{n-i}-\varepsilon_i,\, \beta_j\,\mid\, 1\le i,\,j\le n/2-1\}$.
Recall that $\{\varpi_{i}\}_{1\le i\le n}$ is the set of fundamental weights of $\g$. One has that for all $i$, with $1\le i\le n/2-1$, $\beta_i=\varpi_{2i}-\varpi_{2i-2}$ (where we have set $\varpi_0=0$) and $\varepsilon_n=-\varpi_{n-1}+2\varpi_n$. Also, for all $i$, with $1\le i\le n/2-1$, $\varepsilon_{n-i}-\varepsilon_i=-\varpi_i+\varpi_{i-1}-\varpi_{n-1-i}+\varpi_{n-i}$.

Then, by ordering the basis of $\h_{\Lambda}$ as \par
$\{\alpha_2^{\vee},\,\alpha_4^{\vee},\ldots,\alpha_{n-2}^{\vee},\,\alpha_{n-1}^{\vee},\,\alpha_1^{\vee},\,\alpha_{n-3}^{\vee},\,\alpha_3^{\vee},\ldots, \alpha^\vee_{n/2+1},\alpha^\vee_{n/2-1}\}$ if $n/2$ is even,\par
and as
$\{\alpha_2^{\vee},\,\alpha_4^{\vee},\ldots,\alpha_{n-2}^{\vee},\,\alpha_{n-1}^{\vee},\,\alpha_1^{\vee},\,\alpha_{n-3}^{\vee},\,\alpha_3^{\vee},\ldots, \alpha^\vee_{n/2-2},\alpha^\vee_{n/2}\}$
if $n/2$ is odd, and by ordering elements of $S$ as $\{\beta_1,\,\beta_2,\ldots,\beta_{n/2-1},\varepsilon_n,\,\varepsilon_{n-1}-\varepsilon_1,\varepsilon_{n-2}-\varepsilon_2,\,\ldots,\varepsilon_{n/2+1}-\varepsilon_{n/2-1}\},$ we have that $(s_i(h_j))_{ij}=\begin{pmatrix} A&0\\
C&D\end{pmatrix}$
where $A$ is an $(n/2)\times (n/2)$ lower triangular matrix with $1$ everywhere on the diagonal, except the last element which is equal to $-1$ and $D$ is a $(n/2-1)\times (n/2-1)$ lower triangular matrix with $-1$ everywhere on the diagonal. Hence $\det(s_i(h_j))_{ij}=(-1)^{n/2}\neq 0$.

For every $n\ge s$, let $\{h_1,\,\ldots, h_{n-1},\,h_n\}$ be a basis for the truncated Cartan $\h'_{n+1}$ of the truncated parabolic associated to $\pi_{n+1}\setminus\{\alpha_s\}$ in type $B_{n+1}$, such that
$\{h_1,\,\ldots,\,h_{n-1}\}$ is a basis of the truncated Cartan $\h'_n$ for the truncated parabolic associated to $\pi_n\setminus\{\alpha_s\}$ in type ${\rm B}_n$ with the identification in the beginning of this proof.

Then, using the observation in the beginning of this proof, and ordering the elements of $S_{n+1}=\{s_1,\,s_2,\,,\ldots,\,s_n\}$ such that its first $n-1$ elements are those of $S_n$, we get that
$\det(s_i(h_j))_{1\le i,\,j\le n}=(-1)^n\det(s_i(h_j))_{1\le i,\,j\le n-1}$, which completes the proof of the lemma.
\end{proof}

Recall the (maximal in $\Delta^+$) Heisenberg set $H_{\beta_i}$ of centre $\beta_i$ defined in Example~\ref{MaxHeisenbergset} for every positive  root $\beta_i$ of the Kostant cascade.

\begin{itemize}
\item Set $\Gamma_{\ep_s}=\displaystyle\{\varepsilon_s,\,\varepsilon_i,\,\varepsilon_s-\varepsilon_i,\,\ep_s+\ep_j,\,-\ep_j\,\mid\, 1\le i\le n,\,i\ne s,\, s+1\le j\le n\}\subset\Delta^+\sqcup\Delta^-_{\pi'}.$
\item For all $i\in\mathbb N$, $1\le i\le s/2-1$, set $\Gamma_{\beta_i}=H_{\beta_i}\setminus\{\varepsilon_{2i-1}, \,\varepsilon_{2i}\}\subset\Delta^+.$
\item Set $\Gamma_{\varepsilon_{s-1}+\varepsilon_{s+1}}=\{\varepsilon_{s-1}+\varepsilon_{s+1},\,\ep_{s-1}\pm \ep_{i},\,\ep_{s+1}\mp \ep_i\,\mid \, s+2\le i\le n\}\subset\Delta^+.$
\item For all $i\in\mathbb N$, $s/2+1\le i\le[(n-1)/2]$, set $\Gamma_{\ep_{2i}+\ep_{2i+1}}=\{\ep_{2i}+\ep_{2i+1},\,\ep_{2i}\pm \ep_j,\,\ep_{2i+1}\mp \ep_j\,\mid\, 2i+2\le j\le n\}\subset\Delta^+.$
\item For all $i\in\mathbb N$, $1\le i\le s/2-1$, set $\Gamma_{\ep_{s-i}-\ep_i}=\{\ep_{s-i}-\ep_i,\,\varepsilon_j-\varepsilon_i,\,\varepsilon_{s-i}-\varepsilon_j\mid i+1\le j\le s-i-1\}\subset\Delta^-_{\pi'}.$
\item For all $i\in\mathbb N$, $s/2+1\le i\le [n/2]$, set
$\Gamma_{-\varepsilon_{2i-1}-\varepsilon_{2i}}=\{-\varepsilon_{2i-1}-\varepsilon_{2i},\,-\varepsilon_{2i-1}\pm \varepsilon_{j},\,-\ep_{2i}\mp \ep_j\,\mid\, 2i+1\le j\le n\}\subset\Delta^-_{\pi'}.$
\end{itemize}

By construction, the sets $\Gamma_{\gamma}$, $\gamma\in S$, are disjoint Heisenberg sets of centre $\gamma$, included in $\Delta^+\sqcup\Delta^-_{\pi'}$.

Denote by $\pi'_1$ the connected component of $\pi'$ of type ${\rm A}_{s-1}$ and $\pi'_2$ the connected component
of $\pi'$ of type ${\rm B}_{n-s}$. Observe that for all $i\in\mathbb N$, $1\le i\le s/2-1$, $\Gamma_{\ep_{s-i}-\ep_i}\subset\Delta^-_{\pi'_1}$ and for all $i\in\mathbb N$, 
$s/2+1\le i\le [n/2]$, $\Gamma_{-\ep_{2i-1}-\ep_{2i}}\subset\Delta^-_{\pi'_2}$.

\begin{Remark}\label{rem O}\rm
(1) If $\alpha\in O\setminus O_1\sqcup O_2$ then $\alpha=\ep_i-\ep_j\in O^-$ with $1\le j<s/2<i\le s-1-j$. Moreover in this case, $S_{\ep_i-\ep_j}\cap O^-=\{\theta(\ep_i-\ep_j)\}$ and $S_{\ep_i-\ep_j}\cap O^m=\emptyset$. Hence conditions (3) and (4) of Lemma~\ref{non-degeneracy} are satisfied for such a root. 

(2) For $i,\,j\neq s+1$, $\ep_i+\ep_j\in O_1$ unless $\ep_i+\ep_j=\ep_{s-1}+\ep_s\in T$ (where $T$ is the complement of $\Gamma=\sqcup_{\gamma\in S}\Gamma_{\gamma}$ in $\Delta^+\sqcup\Delta^-_{\pi'}$). Moreover $\ep_s+\ep_{s+1}\in O_1$ and $\ep_s-\ep_i\in O_1$ for all $i\ge s/2,\,i\neq s$.
\end{Remark}

We show below that conditions (2) and (4) for $\alpha\in O^+$, resp. conditions (3) and (4) for $\alpha\in O^-$, of Lemma~\ref{non-degeneracy}, are satisfied.\smallskip

\begin{lemma}\label{Cond B}
Let $\alpha\in O^\pm$. Then $S_{\alpha}\cap O^\pm=\{\theta(\alpha)\}$ and if $S_{\alpha}\cap O^m\neq\emptyset$ then $\alpha\in O_{st}$. \end{lemma}

\begin{proof}
Let $\alpha\in O^\pm$ and assume that there exists $\beta\in O^\pm\sqcup O^m$ such that $\alpha+\beta\in S$, that is, $\beta\in S_{\alpha}$. 

\noindent \underline{First case:} $\alpha\in\Gamma_{\beta_i}^0\subset O^+$, with $1\le i\le s/2-1$.\par
Assume first that $\beta\in \Gamma_{\beta_j}^0\subset \Delta^+$, with $1\le j\le s/2-1$. Then by~\cite[Lemma 3 (5)]{FL} $\alpha+\beta\in H_{\beta_k}$ where $k=\min\{i, j\}$. But the only element of $S$ in $H_{\beta_k}$ is $\beta_k$. By~\cite[Lemma 3 (5)]{FL} again, it follows that $\alpha,\,\beta\in H_{\beta_k}\setminus\{\beta_k\}$. Hence $i=j=k$.

Assume now that $\beta\in O^+\sqcup O^m$ but $\beta\not\in\bigsqcup_{j=1}^{s/2-1}\Gamma_{\beta_j}^0$. If $\beta\in\Delta^+$, either $\beta\in\{\ep_j\mid 1\le j\le s-2\}$ or by~\cite[Lemma 3 (2)]{FL}, $\beta\in\bigsqcup_{j=s/2}^{[(n+1)/2]}H_{\beta_j}\sqcup\bigsqcup_{j'=s/2+1}^{[n/2]}H_{\beta_{j'}}$, where recall that  $H_{\beta_{j'}}=\{\beta_{j'}\}=\{\alpha_{2j'-1}\}$. By a similar reasoning as before or an easy computation one shows that this is not possible. Hence condition (2) for $\alpha$ is satisfied.

Finally suppose that $\beta\in O^m\cap\Delta^-$ and then $S_{\alpha}\cap O^m\neq\emptyset$. One verifies that $\alpha=\ep_j-\ep_s$ and $\beta=\ep_s-\ep_{s-j}$, with $j\in \{2i-1,\, 2i\}$ and $j>s/2$. Moreover one has that $\theta(\alpha)=\ep_{j+1}+\ep_s\in O_1$ if $j$ is odd, resp. $\theta(\alpha)=\ep_{j-1}+\ep_s\in O_1$ if $j$ is even by remark~\ref{rem O} (2). We will assume that $j$ is odd; the other case is very similar.

Recall the sequences of roots in $O$ constructed from a root in $O$ in Section~\ref{sr}. Since $\theta(\alpha)\in O_1$, we have that the sequence $(\alpha^k)_{k\in\mathbb N}$ constructed from $\alpha$ is stationary at rank $0$. We will determine the sequence $(\alpha^{(k)})_{k\in \mathbb N}$ constructed from $\theta(\alpha)$. Recall that $\alpha^{(0)}=\theta(\alpha)$, then since $\alpha=\theta(\alpha^{(0)})\in O_2$, we necessarily have $\alpha^{(1)}=\beta=\ep_s-\ep_{s-j}$. One has that $\theta(\beta)=\ep_{s-j}\in O_2$, with $S_{\ep_{s-j}}=\{\beta,\, \ep_{s-j+1}\}$. Then $\alpha^{(2)}=\ep_{s-j+1}\in O^m$ and $\theta(\alpha^{(2)})=\ep_s-\ep_{s-j+1}$. Then
$\alpha^{(3)}=\ep_{j-1}-\ep_s$ and $\theta(\alpha^{(3)})=\ep_{j-2}+\ep_s\in O_1$ by remark~\ref{rem O} (2) (unless $j=s/2+1$ in which case already $\theta(\alpha^{(2)})\in O_1$). We conclude that the sequence $(\alpha^{(k)})_{k\in \mathbb N}$ is stationary at rank at most $3$. Since $A_{\alpha}\cup A_{\theta(\alpha)}\subset O_1\sqcup O_2$ by remark~\ref{rem O} (1), it follows that $\alpha\in O_{st}$. Hence condition (4) is satisfied for such an $\alpha$.\smallskip

\noindent \underline{Second Case:} $\alpha\in\Gamma_{\ep_{s-1}+\ep_{s+1}}^0\subset O^+$.\par
If $\beta\in\Gamma_{\ep_{s-1}+\ep_{s+1}}^0$ then necessarily $\alpha+\beta=\ep_{s-1}+\ep_{s+1}$ thus $\beta=\theta(\alpha)$.

For condition (2), it remains to check the case where $\beta\in\bigsqcup_{j=s/2+1}^{[(n-1)/2]} \Gamma_{\ep_{2j}+\ep_{2j+1}}^0$. But then it is not possible that $\alpha+\beta\in S$ since $\alpha$ contains $\ep_{s-1}$ or $\ep_{s+1}$ and $\beta$ contains $\ep_i$ with $i\ge s+2$, while in $S^+$ there is no root containing a linear combination of both $\ep_{s-1}$ or $\ep_{s+1}$ and $\ep_i$ with $i\ge s+2$.

Finally, for condition (4) one easily checks that it is not possible that $\beta\in\Gamma_{\ep_s}^0$.\smallskip

\noindent \underline{Third case:} $\alpha\in\Gamma_{\ep_{2i}+\ep_{2i+1}}^0\subset O^+$, with $s/2+1\le i\le[(n-1)/2]$.\par
For condition (2), it remains to check the case where $\beta\in\Gamma_{\ep_{2j}+\ep_{2j+1}}^0$, with $s/2+1\le j\le[(n-1)/2]$.
Then one has that $i=j$ and $\alpha+\beta=\ep_{2i}+\ep_{2i+1}$, thus $\beta=\theta(\alpha)$.

Finally, for condition (4) one checks that it is not possible that $\beta\in\Gamma_{\ep_s}^0$.\smallskip

\noindent \underline{Fourth case:} $\alpha\in\Gamma_{\ep_{s-i}-\ep_i}^0\subset O^-\cap\Delta_{\pi_1'}^-$, with $1\le i\le s/2-1$.\par
Assume first that $\beta\in\Gamma_{\ep_{s-j}-\ep_j}^0$, with $1\le j\le s/2-1$. Since $\Gamma_{\ep_{s-i}-\ep_i}^0\subset\Delta^-_{\pi'_1}$ and $\Gamma_{\ep_{s-j}-\ep_j}^0\subset\Delta^-_{\pi'_1}$, one has that $\alpha+\beta\in\Delta^-_{\pi'_1}$ and so there exists $k$, with $1\le k\le s/2-1$ such that $\alpha+\beta=\ep_{s-k}-\ep_k$. Then necessarily $i=j=k$, thus $\beta=\theta(\alpha)$.

On the other hand, it is not possible that $\beta\in\Gamma_{-\ep_{2j-1}-\ep_{2j}}^0$, with $s/2+1\le j\le[n/2]$, since in that case $\beta\in\Delta^-_{\pi'_2}$ and so $\alpha+\beta\not\in\Delta$. Hence condition (3) for $\alpha$ is satisfied.

Finally one also checks that it is not possible that $\beta\in\Gamma_{\ep_s}^0$.\smallskip

\noindent \underline{Fifth case:} $\alpha\in\Gamma_{-\ep_{2i-1}-\ep_{2i}}^0\subset O^-\cap\Delta_{\pi_2'}^-$, with $s/2+1\le i\le[n/2]$.\par
For condition (3), it remains to check the case where $\beta\in\Gamma_{-\ep_{2j-1}-\ep_{2j}}^0$, with $s/2+1\le j\le[n/2]$ then there exists $k$, with $s/2+1\le k\le[n/2]$ such that
$\alpha+\beta=-\ep_{2k-1}-\ep_{2k}$ (since $\alpha+\beta\in\Delta^-_{\pi'_2}\cap S$) and one checks that
$i=j=k$, thus $\beta=\theta(\alpha)$.

Finally one checks that it is not possible that $\beta\in\Gamma_{\ep_s}^0$.
\end{proof}

We show below that condition (4) of Lemma~\ref{non-degeneracy}
is satisfied for $\alpha\in O^m$.

\begin{lemma}\label{CondbisB}
Let $\alpha\in O^m$. Then $\alpha\in O_{st}$. 
\end{lemma}

\begin{proof}
Assume that $\alpha\in O^m$ ($=\Gamma_{\ep_s}^0$). Note that $S_{\alpha}\cap O^m\neq\emptyset$, since it contains $\theta(\alpha)$. We will show that $\alpha\in O_{st}$ and so condition (4) of Lemma~\ref{non-degeneracy} holds.

Recall that $\Gamma^0_{\ep_s}=\{\ep_i, \,\ep_s-\ep_i,\, -\ep_j,\, \ep_s+\ep_j\,\mid\,1\le i\le n,\,i\neq s,\, s+1\le j\le n\}$. We will show that the sequences of roots in $O$ constructed from the roots in $\Gamma^0_{\ep_s}$ are stationary and that all the elements of these sequences and their image by $\theta$ lie in $O_1\sqcup O_2$. Note that this is enough to prove our claim.

For $\alpha=-\ep_j$, with $s+1\le j\le n$, we have that $\theta(\alpha)\in O_1$ by remark~\ref{rem O} (2), hence the sequence $(\alpha^i)$ is stationary at rank $0$.

For $\alpha=\ep_s+\ep_j$, with $s+1\le j\le n$, we have $\theta(\alpha)\in O_2$ and $\alpha^1=-\ep_{j+1}$ if $j$ is odd, resp. $\alpha^1=-\ep_{j-1}$ if $j$ is even. Then $\theta(\alpha^1)\in O_1$, hence $(\alpha^i)$ is stationary at rank $1$.

For $\alpha=\ep_i$ with $i\ge s+2$ then $\theta(\alpha)\in O_1$. Also for $\alpha=\ep_s-\ep_i$ and $i\ge s+2$, $\theta(\alpha)\in O_2$ and $\alpha^1=\ep_{i+1}$ if $i$ even, $\alpha^1=\ep_{i-1}$ if $i$ odd and $\theta(\alpha^1)\in O_1$. We conclude as above.

For $\alpha=\ep_{s\pm 1}$, then $\theta(\alpha)\in O_1$ and we are done. For $\alpha=\ep_s-\ep_{s\pm 1}$ then $\theta(\alpha)=\ep_{s\pm 1}\in O_2$ and $\alpha^1=\ep_{s\mp 1}$ is such that $\theta(\alpha^1)=\ep_s-\ep_{s\mp 1}\in O_1$.

For $\alpha=\ep_{i}$ with $1\le i\le s-2$, $\theta(\alpha)\in O_1$ if $i\ge s/2$,
 otherwise $\theta(\alpha)\in O_2$.  In the latter case, by the first case of the proof of Lemma~\ref{Cond B} (last part), we obtain that
the sequence of roots in $O$ constructed from $\alpha$ is stationary.

It remains to consider $\alpha=\ep_s-\ep_i$, with $1\le i\le s-2$. Then $\theta(\alpha)=\ep_i\in O_2$ and $\alpha^1=\ep_{i+1}$ if $i$ is odd, $\alpha^1=\ep_{i-1}$ if $i$ is even. By the above, $\theta(\alpha^1)\in O_1$ if $i>s/2$ or $i=s/2$ and $s/2$ odd and we are done. In the other cases, $\theta(\alpha^1)\in O_2$ by the above, which also gives that the sequence of roots in $O$ constructed from $\alpha^1$ and then from $\alpha$ is stationary.

Finally we observe that all roots of the sequences and their image by $\theta$ lie in $O_1\sqcup O_2$.
\end{proof}

Now denote by $T$ the complement of $\Gamma=\Gamma^+\sqcup\Gamma^-\sqcup\Gamma^m=\bigsqcup_{\gamma\in S}\Gamma_{\gamma}$ in $\Delta^+\sqcup\Delta^-_{\pi'}$.

\begin{lemma}\label{setTB}
One has that $\lvert T\rvert=\ind\,\p.$
\end{lemma}

\begin{proof} One checks that:
\begin{itemize}
\item For $n=s$, $T=\{\ep_{s-1}+\ep_s,\,\ep_{2i-1}-\ep_{2i}\,\mid\, 1\le i\le s/2\}$. 
\item For $n>s,\, T=\{\ep_{s-1}+\ep_s,\,\ep_{s-1}-\ep_{s+1},\,\ep_{2i-1}-\ep_{2i},\,-\ep_{s+2j-1}+\ep_{s+2j},\, \ep_{s+2k}-\ep_{s+2k+1}\,|\,1\le i\le s/2,\, 1\le j\le [(n-s)/2],\, 1\le k\le [(n-s-1)/2]\}$.
\end{itemize}

From the above description of $T$, it follows that $\lvert T\rvert=n-s/2+1$. On the other hand, recall that the index of $\p$ equals the number of $\langle {\bf ij}\rangle$-orbits in $\pi$
where ${\bf i}$ and ${\bf j}$ are the involutions of $\pi$ of Section~\ref{standardnotation}.

Here the $\langle {\bf ij}\rangle$-orbits in $\pi$ are $\Gamma_t=\{\alpha_t,\,\alpha_{s-t}\}$ for $1\le t\le s/2-1$,
$\Gamma_{s/2}=\{\alpha_{s/2}\}$ and $\Gamma_t=\{\alpha_t\}$ for $s\le t\le n$. They are $n-s/2+1$ in number
hence $\ind\,\p=n-s/2+1$.\end{proof}

\begin{Remark}\rm
All conditions of Lemma~\ref{refineregular} are satisfied. Hence by defining $h\in\h_{\Lambda}$ by $\gamma(h)=-1$, for all $\gamma\in S$
and by setting $y=\sum_{\gamma\in S}x_{\gamma}$ we obtain an adapted pair $(h,\,y)$ for $\p^-_{\pi',\,\Lambda}$.
\end{Remark}

 The semisimple element $h$ of the adapted pair is uniquely defined by the relations $\gamma(h)=-1$ for all $\gamma\in S$. Below we compute the values of $h$ on the elements of $T$, that is the $\ad h$ eigenvalues on the complement $\g_T$ of the $\ad\p^-_{\pi',\,\Lambda}$-orbit of $y$.

\begin{lemma}\label{eigenvalues B} The semisimple element $h$ of the above adapted pair $(h,\,y)$
for $\p^-_{\pi',\,\Lambda}$ is
$$\begin{array}{lll}
h&=&\sum_{k=1}^{[s/4]}\bigl(\frac{s}{2}+2k-1\bigr)\ep_{2k-1}+\sum_{k=[s/4]+1}^{s/2-1}\bigl(\frac{3s}{2}-2k\bigr)\ep_{2k-1}\\
&-&\sum_{k=1}^{[s/4]}\bigl(\frac{s}{2}+2k\bigr)\ep_{2k}
-\sum_{k=[s/4]+1}^{s/2-1}\bigl(\frac{3s}{2}+1-2k\bigr)\ep_{2k}+\frac{s}{2}\ep_{s-1}-\ep_s\\
&+&\sum_{k=1}^{[(n-s+1)/2]}\bigl(-2k+1-\frac{s}{2}\bigr)\ep_{s+2k-1}+\sum_{k=1}^{[(n-s)/2]}\bigl(2k+\frac{s}{2}\bigr)\ep_{s+2k}.\end{array}$$

Then the eigenvalues of $\ad h$ on $\g_T$ are :
\begin{itemize}
\item $s+4i-1=(\ep_{2i-1}-\ep_{2i})(h)$ for all $i\in\mathbb N$, $ 1\le i\le [s/4]$.
\item $3s-4i+1=(\ep_{2i-1}-\ep_{2i})(h)$ for all $i\in\mathbb N$, $[s/4]+1\le i\le s/2-1$.
\item $s/2+1=(\ep_{s-1}-\ep_s)(h)$.
\item $s/2-1=(\ep_{s-1}+\ep_s)(h)$.
\item $s+1=(\ep_{s-1}-\ep_{s+1})(h)$.
\item $s+4j-1=(-\ep_{s+2j-1}+\ep_{s+2j})(h)$, for all $j\in\mathbb N$, $1\le j\le [(n-s)/2]$.
\item $s+4j+1=(\ep_{s+2j}-\ep_{s+2j+1})(h)$, for all $j\in\mathbb N$, $1\le j\le [(n-s-1)/2]$.
\end{itemize}
From the last three equalities we have that $s+2k-1$ is an eigenvalue of $\ad h$ on $\g_T$, for all $k\in\mathbb N$, $1\le k\le n-s$.\end{lemma}
\begin{proof}
Follows by a direct computation.
\end{proof}

Recall the bounds $\ch\mathcal A$ and $\ch\mathcal B$ for $\ch Y(\p)$  as well as the improved upper bound $\mathcal B'$  of Section~\ref{standardnotation}. We will show that the lower bound $\ch\mathcal A$ and the improved upper bound $\mathcal B'$ coincide, hence $Y(\p)$ is a polynomial algebra over $k$.

\begin{lemma}\label{calculboundsB}
If $n=s$ one has that
$$\ch\mathcal A=(1-e^{-2\varpi_n})^{-2}(1-e^{-4\varpi_n})^{-(n/2-1)}\eqno (1)$$

If $n>s$, one has that
$$\ch\mathcal A=(1-e^{-\varpi_s})^{-2}(1-e^{-2\varpi_s})^{-(n-1-s/2)}\eqno (2)$$
\end{lemma}
\begin{proof}
The lower bound for $\ch Y(\p)$ is $$\ch\mathcal A=\prod_{\Gamma\in E(\pi')}(1-e^{\delta_{\Gamma}})^{-1}\le \ch\, Y(\p).$$ We will compute it explicitly. As we already said in the proof of Lemma~\ref{setTB}, the set of $\langle\bf ij\rangle$-orbits in $\pi$ is $$E(\pi')=\{\Gamma_{s/2}:=\{\alpha_{s/2}\},\, \Gamma_t:=\{\alpha_t,\,\alpha_{s-t}\},\,\Gamma_u:=\{\alpha_u\}\mid 1\le t\le s/2-1,\,s\le u\le n\}.$$
It remains to compute $\delta_{\Gamma}$ for each $\Gamma\in E(\pi')$.

Let $\Gamma\in E(\pi')$. Since ${\bf j}={\rm id}_{\pi}$ and ${\bf i}(\Gamma\cap\pi')={\bf j}(\Gamma)\cap\pi'$, one has
$$\delta_{\Gamma}=-2(\sum_{\gamma\in\Gamma}\varpi_{\gamma}-\sum_{\gamma\in\Gamma\cap\pi'}\varpi'_{\gamma}).$$

Assume first that $n=s$. Then the Levi factor of $\p$ is of type ${\rm A}_{n-1}$ and one may check that
for all $1\le t\le n-1$, $\varpi_t-\varpi'_t=\displaystyle 2(t/n)\varpi_n$. Then for all $1\le t\le n/2-1$, $\delta_{\Gamma_t}=-2(\varpi_t-\varpi'_t+\varpi_{n-t}-\varpi'_{n-t})=-4\varpi_n$
and $\delta_{\Gamma_{n}} =\delta_{\Gamma_{n/2}}=-2\varpi_n$. Hence for $n=s$, one has the equality (1).

Assume now that $n>s$. Then the Levi factor of $\p$ is the product of a simple Lie algebra of type ${\rm A}_{s-1}$
and a simple Lie algebra of type ${\rm B}_{n-s}$.

For all $1\le t\le s-1$, one checks that $\varpi_t-\varpi'_t=(t/s)\varpi_s$. Then, for all $1\le t\le s/2-1$, one has $\delta_{\Gamma_t}=-2\varpi_s$ and $\delta_{\Gamma_{s/2}}=-\varpi_s$. On the other hand, for all $s+1\le t\le n-1$, one has that $\varpi_t-\varpi'_t=\varpi_s$, hence $\delta_{\Gamma_t}=-2\varpi_s$. Finally $\varpi_n-\varpi'_n=(1/2)\varpi_s$ and $\delta_{\Gamma_n}=-\varpi_s$, whereas $\delta_{\Gamma_s}=-2\varpi_s$, since $\Gamma_s\cap \pi'=\emptyset$.

We conclude that for $n>s$, one has the equality (2).
\end{proof}

\begin{lemma}\label{improvboundB}

The bound $\mathcal B'$ is given by the right hand side of (1) if $n=s$, resp. of (2) if $n>s$. Hence one has that $\ch\mathcal A=\mathcal B'$ and then $Y(\p)$ is a polynomial algebra over $k$.

\end{lemma}
\begin{proof}
Recall Section~\ref{standardnotation} that the improved upper bound for $\ch Y(\p)$ is

$$\mathcal B'= \prod_{\gamma\in T}(1-e^{-(\gamma+t(\gamma))})^{-1},$$
where for all $\gamma\in T$, $t(\gamma)$ is the unique element in $\mathbb Q S$ such that $\gamma+t(\gamma)$ is a multiple of $\varpi_s$. We will compute $t(\gamma)$, for all $\gamma\in T$.

Assume first that $n=s$ and recall  that $T=\{\ep_{s-1}+\ep_s,\,\ep_{2i-1}-\ep_{2i}\,\mid\, 1\le i\le s/2\}$. Recall also  that $S=\{\varepsilon_s,\,\varepsilon_{s-i}-\varepsilon_i,\,\varepsilon_{2j-1}+\varepsilon_{2j}\,\mid\, 1\le i,\,j\le s/2-1\}$. Finally recall that $\varpi_s=\varpi_n=1/2(\ep_1+\ep_2+\cdots+\ep_n)$.

By a direct calculation, one may verify that:\par
$\bullet$ $t(\varepsilon_{s-1}+\varepsilon_{s})=(\varepsilon_1+\varepsilon_2)+(\ep_3+\ep_4)+\ldots+(\ep_{n-3}+\ep_{n-2})$ and $\varepsilon_{s-1}+\varepsilon_{s}+t(\varepsilon_{s-1}+\varepsilon_{s})=2\varpi_n$.\par
$\bullet$  $t(\ep_{s-1}-\ep_s)=(\ep_1+\ep_2)+\ldots+(\ep_{n-3}+\ep_{n-2})+2\ep_n$ and $\ep_{s-1}-\ep_s+t(\ep_{s-1}-\ep_s)=2\varpi_n$.\par
$\bullet$ For $1\le i\le s/2-1$:
\begin{enumerate}
\item If $n\le 4i-2$, $$\begin{array}{ll}  t(\ep_{2i-1}-\ep_{2i})&=\displaystyle 2\sum_{j=1}^{n-2i}(\ep_{n-j}-\ep_{j})+4\sum_{j=1}^{n/2-i}(\ep_{2j-1}+\ep_{2j})\\
&\displaystyle+2\sum_{j=n/2-i+1}^{i-1}(\ep_{2j-1}+\ep_{2j})+
(\ep_{2i-1}+\ep_{2i})+2\ep_n\end{array}$$ and $(\ep_{2i-1}-\ep_{2i})+t(\ep_{2i-1}-\ep_{2i})=4\varpi_n$.
\item If $n>4i-2$, $$\begin{array}{ll}
t(\ep_{2i-1}-\ep_{2i})&=\displaystyle 2\sum_{j=1}^{2i-1}(\ep_{n-j}-\ep_{j})+4\sum_{j=1}^{i-1}(\ep_{2j-1}+\ep_{2j})\\
&\displaystyle+2\sum_{j=i+1}^{n/2-i}(\ep_{2j-1}+\ep_{2j})+
3(\ep_{2i-1}+\ep_{2i})+2\ep_n
\end{array}$$\par
 and $(\ep_{2i-1}-\ep_{2i})+t(\ep_{2i-1}-\ep_{2i})=4\varpi_n$.
\end{enumerate}

Hence for all $1\le i\le s/2-1$, $\ep_{2i-1}-\ep_{2i}+t(\ep_{2i-1}-\ep_{2i})=4\varpi_n$.\par
We conclude that when $n=s$ the product $\prod_{\gamma\in T}(1-e^{-(\gamma+t(\gamma))})^{-1}$ is given by the right hand side of equality (1) of Lemma~\ref{calculboundsB} and hence coincides with the lower bound for $\ch Y(\p)$.

Now assume that $n>s$. The previous computations hold if we replace $n$ by $s$ and $2\varpi_n$ by $\varpi_s$ (and so $4\varpi_n$
by $2\varpi_s$). Then we may recover $t(\gamma)$  and $\gamma+t(\gamma)$ for $\gamma=\varepsilon_{s-1}+\varepsilon_{s}$, $\gamma=\ep_{s-1}-\ep_s$ or $\gamma=\ep_{2i-1}-\ep_{2i}$, $1\le i\le s/2-1$, by the above.

It remains to compute $t(\gamma),\, \gamma+t(\gamma)$ for the rest of the elements in $T$.\par
$\bullet$ $t(\ep_{s-1}-\ep_{s+1})=2((\ep_1+\ep_2)+\ldots+(\ep_{s-3}+\ep_{s-2}))+(\ep_{s-1}+\ep_{s+1})+2\ep_s$ and $(\ep_{s-1}-\ep_{s+1})+t(\ep_{s-1}-\ep_{s+1})=2\varpi_s$.\par
$\bullet$ For $1\le j\le [(n-s)/2]$,
$$\begin{array}{lll}
t(-\ep_{s+2j-1}+\ep_{s+2j})&=\displaystyle 2((\ep_1+\ep_2)+\ldots+(\ep_{s-3}+\ep_{s-2}))+2(\ep_{s-1}+\ep_{s+1})\\
&\displaystyle -2\sum_{k=1}^{j-1}(\ep_{s+2k-1}+\ep_{s+2k})+2\sum_{k=1}^{j-1}(\ep_{s+2k}+\ep_{s+2k+1})\\
&-(\ep_{s+2j-1}+\ep_{s+2j})+2\ep_s
\end{array}$$\par
and $(-\ep_{s+2j-1}+\ep_{s+2j})+t(-\ep_{s+2j-1}+\ep_{s+2j})=2\varpi_s$.\par
$\bullet$ For $1\le j\le[(n-s-1)/2]$,
$$\begin{array}{lll}
t(\ep_{s+2j}-\ep_{s+2j+1})&=\displaystyle 2((\ep_1+\ep_2)+\ldots+(\ep_{s-3}+\ep_{s-2}))+2(\ep_{s-1}+\ep_{s+1})\\
&\displaystyle -2\sum_{k=1}^{j}(\ep_{s+2k-1}+\ep_{s+2k})+2\sum_{k=1}^{j-1}(\ep_{s+2k}+\ep_{s+2k+1})\\
&+(\ep_{s+2j}+\ep_{s+2j+1})+2\ep_s
\end{array}$$\par
and $(\ep_{s+2j}-\ep_{s+2j+1})+t(\ep_{s+2j}-\ep_{s+2j+1})=2\varpi_s$.\par
We conclude that also for $n>s$ the product $\prod_{\gamma\in T}(1-e^{-(\gamma+t(\gamma))})^{-1}$ is given by the right hand side of equality (2) of Lemma~\ref{calculboundsB} and hence coincides with the lower bound for $\ch Y(\p)$.
\end{proof}

%\subsubsection{Conclusion.}Recall~\ref{improvedupperbound} and~\ref{degrees}.

\begin{theorem}\label{conB} 
Let $\g$ be a simple Lie algebra of type ${\rm B}_n$, $n\ge 2$, and let $\p=\p^-_{\pi',\,\Lambda}$ be a  truncated maximal parabolic
subalgebra of $\g$ associated to $\pi'=\pi\setminus\{\alpha_s\}$, where $s$ is an even integer, $s\le n$.

There exists an adapted pair $(h,\,y)$ for $\p$ and an affine slice $y+\g_T$ in $\p^*$ such that restriction of functions gives an isomorphism of algebras between $Y(\p)$ and the ring $R[y+\g_T]$ of polynomial functions on $y+\g_T$.

In particular $Y(\p)$ is a polynomial algebra over $k$ and the field $C(\p^-_{\pi'})$ of invariant fractions is a purely transcendental extension of $k$.
\end{theorem}

\begin{proof}
Follows by the previous Lemma and by what we said at the end of  Section~\ref{standardnotation}.
\end{proof}

\begin{Remark} \rm (1) In the particular case $s=2$ polynomiality was known by~\cite{PPY} and an adapted pair was constructed in~\cite{J6bis}. Our adapted pair is equivalent to the adapted pair of Joseph $(h',\, y'=\sum_{s\in S'}x_s)$, in the sense of~\cite[2.1.1]{FJ4}. Indeed one verifies that $w=\prod_{k=1}^{[(n-1)/2]}r_{\ep_{2k+1}}\circ r_{\alpha_1}\in W_{\pi'}$ and sends bijectively $S$ to $S'$.\par

(2) The degrees of a set of homogeneous generators of $Y(\p)$ are equal to the eigenvalues of $\ad h$ on $\g_T$ computed in Lemma~\ref{eigenvalues B} each augmented by $1$.
\end{Remark}

\section{Type D, non-extremal case.}\label{SecD}

In this section, the Lie algebra $\g$ is a simple Lie algebra of type ${\rm D}_n$ ($n\ge 4$) and we consider the truncated  maximal parabolic subalgebra $\p=\p^-_{\pi',\,\Lambda}$ associated to $\pi'=\pi\setminus\{\alpha_s\}$ with $s$ even, $2\le s\le n-2$. By~\cite[5.1]{FL} the lower and upper bounds of Section~\ref{standardnotation} for $\ch Y(\p)$ do not coincide. We will construct an adapted pair $(h,\,y)$ for $\p$ and show that the algebra $Y(\p)$ is a polynomial algebra over $k$.

Let $\{\ep_i\}_{1\le i\le n}$ be an orthonormal basis for $\mathbb R^n$ that is used to expand all simple roots $\alpha_i$ ($1\le i\le n$) of $\pi$ as in~\cite[Planche IV]{BOU}.

Recall the Kostant cascade formed by the strongly orthogonal positive roots $\beta_i$, $\beta_{i'}$, $\beta_{i''}$ given in~\cite[Table I]{FL} or in~\cite[Table II]{J1}:
  note that in~\cite[Table I]{FL}, we had forgotten $\beta_{(n+1)/2}$ for $n$ odd:

-- for all $1\le i\le[n/2]$, $\beta_i=\ep_{2i-1}+\ep_{2i}$, and
if $n$ is odd, $\beta_{(n+1)/2}=\alpha_{n-2}=\ep_{n-2}-\ep_{n-1}$,

-- for all $1\le i'\le[n/2]-1$, $\beta_{i'}=\ep_{2i'-1}-\ep_{2i'}$,

-- if $n$ is even, $\beta_{(\frac{n-2}{2})''}=\alpha_{n-1}=\ep_{n-1}-\ep_{n}$.

Set $S=S^+\sqcup S^-\sqcup S^m$ with \begin{itemize}
\item $S^+=\{\ep_{2i-1}+\ep_{2i},\,\ep_{s-1}+\ep_{s+1},\,\ep_{2j}+\ep_{2j+1}\,\mid\, 1\le i\le s/2-1,\,s/2+1\le j\le[(n-2)/2]\},$
\item $S^-=\{\ep_{s-i}-\ep_i,\,-\ep_{2j-1}-\ep_{2j}\,\mid\, 1\le i\le s/2-1,\,s/2+1\le j\le[(n-1)/2]\}$,
 \item $S^m=\{\ep_s-\ep_n,\,\ep_s+\ep_n\}$.
 \end{itemize}

Clearly one has that $S\subset\Delta^+\sqcup\Delta^-_{\pi'}$ and $\lvert S\rvert=n-1=\dim\h_{\Lambda}$.

Observe that $S$ in type ${\rm D}_n$ is almost identical with the set $S$ in type ${\rm B}_n$.
We first show below that condition (1) of Lemma~\ref{non-degeneracy} holds.

\begin{lemma}
$S_{\mid\h_{\Lambda}}$ is a basis for $\h_{\Lambda}^*$.
\end{lemma}

\begin{proof}
Set $S=\{s_i\}_{1\le i\le n-1}$ with $s_{n-2}=\ep_s-\ep_n$ and $s_{n-1}=\ep_s+\ep_n$ and choose $\{h_i\}_{1\le i\le n-1}=\{\alpha_i^\vee\}_{1\le i\le n,\,i\neq s}$ as a basis of $\h_{\Lambda}$.

Denote by $s'_{n-2}=\ep_s$ and $s'_{n-1}=\ep_n$ and $s'_i=s_i$ for all $1\le i\le n-3$ and set $S'=\{s'_i\}_{1\le i\le n-1}$. It is sufficient to prove that $\det(s'_i(h_j))_{1\le i,\,j\le n-1}\neq 0$.

By ordering the basis of $\h_{\Lambda}$ as $\{\alpha_{2i}^{\vee},\,\alpha_{s-1}^{\vee},\,\alpha_{2j-1}^{\vee},\,\alpha_{s-2j-1}^{\vee},\,\alpha_{k}^{\vee}\mid 1\le i\le s/2-1,\,1\le j\le [s/4],\,\,s+1\le k\le n\}$ without repetitions
and the elements of $S'$ as $\{\beta_i,\ep_s,\,\ep_{s-i}-\ep_i,\ep_{s-1}+\ep_{s+1},\,-\ep_{s+2j-1}-\ep_{s+2j},\,\ep_{s+2j}+\ep_{s+2j+1},\,\ep_{n}\mid 1\le i\le s/2-1,\,1\le j\le (n-s-2)/2\}$ if $n$ is even
and
$\{\beta_i,\,\ep_s,\,\ep_{s-i}-\ep_i,\,\ep_{s-1}+\ep_{s+1},\,-\ep_{s+2j-1}-\ep_{s+2j},\,\ep_{s+2j}+\ep_{s+2j+1},\,-\ep_{n-2}-\ep_{n-1},\,\ep_{n}\mid 1\le i\le s/2-1,\,1\le j\le (n-s-3)/2
\}$ if $n$ is odd,
one checks that $(s'_i(h_j))_{1\le i,\,j\le n-1}=\begin{pmatrix}
A&0&0\\
*&B&0\\
*&*&C
\end{pmatrix}$
where  $A$ (resp. $B$) is a  $(s/2-1)\times(s/2-1)$ (resp. a  $(s/2)\times(s/2)$) lower triangular matrix with $1$ (resp. $-1$) on the diagonal.
Moreover $C=\begin{pmatrix}
C'&0\\
*&C''\\
\end{pmatrix}$
with $C'$ an  $(n-s-2)\times(n-s-2)$ lower triangular matrix with alternating $1$ and $-1$ on the diagonal and $C''$ a $2\times 2$ matrix. Then $\det(C')=(-1)^{[(n-s-2)/2]}$ and  $\det(C'')=(-1)^{n-s}\times 2$.
We conclude that $\det(s'_i(h_j))_{1\le i,\,j\le n-1}\neq 0$,
which completes the proof of the lemma.
\end{proof}

For each $\gamma\in S$, we will define the Heisenberg set
$\Gamma_{\gamma}$. Set $\Gamma^\pm=\bigsqcup_{\gamma\in S^\pm}\Gamma_{\gamma}$ and $\Gamma^m=\bigsqcup_{\gamma\in S^m}\Gamma_{\gamma}$.

\begin{itemize}
\item For all $i\in\mathbb N$, $1\le i\le s/2-1$, set $\Gamma_{\beta_i}=H_{\beta_i}\setminus\{\ep_{2i-1}-\ep_n,\,\ep_{2i}+\ep_n\}$ where $H_{\beta_i}$ was defined in Example~\ref{MaxHeisenbergset}. 
\item Set $\Gamma_{\ep_{s-1}+\ep_{s+1}}=\{\ep_{s-1}+\ep_{s+1},\,\ep_{s-1}+\ep_i,\,\ep_{s+1}-\ep_i,\,\ep_{s-1}-\ep_j,\,\ep_{s+1}+\ep_j\mid s+2\le i\le n,\,s+2\le j\le n-1\}.$
\item For all $i\in\mathbb N$, $s/2+1\le i\le[(n-2)/2]$, set
$\Gamma_{\ep_{2i}+\ep_{2i+1}}=\{\ep_{2i}+\ep_{2i+1},\,\ep_{2i}-\ep_j,\,\ep_j+\ep_{2i+1},\,\ep_{2i}+\ep_k,\,\ep_{2i+1}-\ep_k\mid 2i+2\le j\le n,\, 2i+2\le k\le n-1\}.$

\item For all $i\in\mathbb N$, $1\le i\le s/2-1$, set $\Gamma_{\ep_{s-i}-\ep_i}=\{\ep_{s-i}-\ep_i,\,\varepsilon_j-\varepsilon_i,\,\varepsilon_{s-i}-\varepsilon_j\mid i+1\le j\le s-i-1\}.$
\item For all $i\in\mathbb N$, $s/2+1\le i\le [(n-1)/2]$, set 
$\Gamma_{-\ep_{2i-1}-\ep_{2i}}=\{-\ep_{2i-1}-\ep_{2i},\,-\ep_{2i-1}-\ep_j\;,\;\ep_j-\ep_{2i},\,-\ep_{2i-1}+\ep_k,\,-\ep_k-\ep_{2i}\mid 2i+1\le j\le n-1,\,2i+1\le k\le n\}.$

\item Set 
$\Gamma_{\ep_s-\ep_n}=\{\ep_s-\ep_n,\,\ep_s-\ep_{2i-1},\,\ep_{2i-1}-\ep_n,\,\ep_s+\ep_{2j+1},\,-\ep_{2j+1}-\ep_n\mid 1\le i\le [n/2],\,i\neq s/2+1,\,s/2\le j\le [(n-2)/2]\}$.
\item Set $\Gamma_{\ep_s+\ep_n}=\{\ep_s+\ep_n,\,\ep_s-\ep_{2i},\,\ep_{2i}+\ep_n,\,\ep_s-\ep_{s+1},\,\ep_{s+1}+\ep_n,\,\ep_s+\ep_{2j},\,-\ep_{2j}+\ep_n\mid 1\le i\le [(n-1)/2],\,i\neq s/2,\,s/2+1\le j\le[( n-1)/2]\}.$
\end{itemize}

By construction, the sets $\Gamma_{\gamma}$, $\gamma\in S$, are disjoint Heisenberg sets of centre $\gamma$, included in $\Delta^+\sqcup\Delta^-_{\pi'}$.

We have that $\Gamma^+=\bigsqcup_{1\le i\le s/2-1}\Gamma_{\beta_i}\sqcup\Gamma_{\ep_{s-1}+\ep_{s+1}}\sqcup\bigsqcup_{s/2+1\le i\le [(n-2)/2]}\Gamma_{\ep_{2i}+\ep_{2i+1}}$, that
$\Gamma^-=\bigsqcup_{1\le i\le s/2-1}\Gamma_{\ep_{s-i}-\ep_i}\sqcup\bigsqcup_{s/2+1\le i\le[(n-1)/2]}\Gamma_{-\ep_{2i-1}-\ep_{2i}}$ and $\Gamma^m=\Gamma_{\ep_s-\ep_n}\sqcup\Gamma_{\ep_s+\ep_n}$.

We show below that conditions (2), (3), (4) of Lemma~\ref{non-degeneracy} are satisfied.
Denote by $\pi'_1$ the connected component of $\pi'$ of type ${\rm A}_{s-1}$ and by $\pi'_2$ the connected component of $\pi'$ of type ${\rm D}_{n-s}$ (or $A_1\times A_1$ if $s=n-2$).

\begin{lemma}\label{Cond D}
Let $\alpha\in O^\pm$. Then $S_{\alpha}\cap O^\pm=\{\theta(\alpha)\}$ and if $S_{\alpha}\cap O^m\neq\emptyset$ then $\alpha\in O_{st}$ or $\alpha\in O_{cyc}$ or there exists $\beta\in O_{cyc}\cap O_3$ such that $\alpha=\tilde \beta$ or $\theta(\alpha)=\tilde\beta$, where $\tilde\beta\in S_{\beta}\cap O_2$ is such that $\theta(\tilde\beta)\in O_1$. \end{lemma}

\begin{proof}

Let $\alpha\in O^\pm$ and assume that there exists $\beta\in O^\pm\sqcup O^m$ such that $\alpha+\beta\in S$.

\underline{First case} : $\alpha\in\Gamma^0_{\beta_i}\subset O^+$ with $1\le i\le s/2-1$. As in the proof of Lemma~\ref{Cond B} one checks that, if $\beta\in O^+\sqcup(O^m\cap\Delta^+)$, then $\beta=\theta(\alpha)$. 
Assume now that $\beta\in O^m\cap\Delta^-$ which implies that $S_{\alpha}\cap O^m\neq\emptyset$. Then four cases occur : $\alpha=\ep_{2i}-\ep_n$ and $\beta=\ep_s-\ep_{2i}\in\Gamma^0_{\ep_s+\ep_n}$, $\alpha=\ep_{2i-1}+\ep_n$ and $\beta=\ep_s-\ep_{2i-1}\in\Gamma^0_{\ep_s-\ep_n}$, $\alpha=\ep_{2i-1}-\ep_s$ and $\beta=\ep_s-\ep_{s-2i+1}$ with $s-2i+1<2i-1$, $\alpha=\ep_{2i}-\ep_s$ and $\beta=\ep_s-\ep_{s-2i}$ with $s-2i<2i$.

Let consider just one of the two first cases, that is when $\alpha=\ep_{2i-1}+\ep_n$ and $\beta=\ep_s-\ep_{2i-1}\in\Gamma^0_{\ep_s-\ep_n}$.
Then $\alpha+\beta=\ep_s+\ep_n$, $\theta(\alpha)=\ep_{2i}-\ep_n$ and $\theta(\beta)=\ep_{2i-1}-\ep_n$. One verifies that there exists $\gamma=\ep_s-\ep_{2i}\in\Gamma^0_{\ep_s+\ep_n}$ such that $\theta(\alpha)+\gamma=\ep_s-\ep_n$, $\theta(\beta)+\theta(\gamma)=\ep_{2i-1}+\ep_{2i}$ and that $\alpha,\,\theta(\alpha),\theta(\beta),\,\theta(\gamma)\in O_2$. If $i=1$ or $s-2i+1\le 2i-1$ (resp. $s-2i\le 2i$) then $\beta\in O_2$ (resp. $\gamma\in O_2$). Otherwise $\beta\in O_3$, $\tilde\beta=\ep_{s-2i+1}-\ep_s\in O_2\cap S_{\beta}$ and $\theta(\tilde\beta)=\ep_{s-2i+2}+\ep_s\in O_1$ (resp. $\gamma\in O_3$, $\tilde\gamma=\ep_{s-2i}-\ep_s\in O_2\cap S_{\gamma}$ and $\theta(\tilde\gamma)=\ep_{s-2i-1}+\ep_s\in O_1$). Hence $\alpha\in O_{cyc}$ and by remark~\ref{remcyc} (1), the roots $\beta,\,\gamma,\,\theta(\alpha),\,\theta(\beta),\,\theta(\gamma)$ are also cyclic roots.

Let consider just one of the two last cases. Suppose that $\alpha=\ep_{2i-1}-\ep_s$ and $\beta=\ep_s-\ep_{s-2i+1}$ with $s-2i+1<2i-1$.
By the above, $\beta\in O_{cyc}\cap O_3$ and $\tilde\beta=\alpha$. 

\underline{Second case} : $\alpha\in\Gamma^0_{\ep_{s-1}+\ep_{s+1}}\subset O^+$. One easily checks that if $\beta\in O^+$ then $\beta=\theta(\alpha)$.  Now if $\beta\in\Gamma^0_{\ep_s-\ep_n}$ then necessarily $\alpha=\ep_{s-1}+\ep_n$,
$\beta=\ep_s-\ep_{s-1}$ and $\alpha+\beta=\ep_s+\ep_n$. One checks that $\gamma=\ep_s-\ep_{s+1}\in\Gamma^0_{\ep_s+\ep_n}$ verifies $\theta(\alpha)+\gamma=\ep_s-\ep_n$
and $\theta(\beta)+\theta(\gamma)=\ep_{s-1}+\ep_{s+1}$. Moreover $\alpha,\,\theta(\alpha),\,\beta,\,\theta(\beta),\,\gamma,\,\theta(\gamma)\in O_2$. Hence $\alpha\in O_{cyc}$. A similar computation shows that if $\beta\in\Gamma^0_{\ep_s+\ep_n}$, then $\alpha\in O_{cyc}$. 

\underline{Third case} : $\alpha\in\Gamma^0_{\ep_{2i}+\ep_{2i+1}}\subset O^+$ with $s/2+1\le i\le [(n-2)/2]$. One easily checks that if $\beta\in O^+$ then $\beta=\theta(\alpha)$. 

Assume now that $\beta\in\Gamma^0_{\ep_s-\ep_n}$. Then necessarily $\alpha=\ep_{2i+1}+\ep_n$, $\beta=\ep_s-\ep_{2i+1}$ and $\alpha+\beta=\ep_s+\ep_n$. Then $\gamma=\ep_s-\ep_{2i}\in\Gamma^0_{\ep_s+\ep_n}$ verifies $\theta(\alpha)+\gamma=\ep_s-\ep_n$ and $\theta(\beta)+\theta(\gamma)=\ep_{2i}+\ep_{2i+1}$. Moreover one has that
$\alpha,\,\theta(\alpha),\,\beta,\,\theta(\beta),\,\gamma,\,\theta(\gamma)\in O_2$. Hence $\alpha\in O_{cyc}$. A similar computation shows that if $\beta\in\Gamma^0_{\ep_s+\ep_n}$, then $\alpha\in O_{cyc}$. 

\underline{Fourth case} : $\alpha\in\Gamma^0_{\ep_{s-i}-\ep_i}\subset O^-$ with $1\le i\le s/2-1$. If $\beta\in\Gamma^0_{\ep_{s-j}-\ep_j}$ with $1\le j\le s/2-1$ then one checks that $j=i$ and $\alpha+\beta=\ep_{s-i}-\ep_i$ thus $\beta=\theta(\alpha)$. Moreover one checks that it is not possible that $\beta\in\Gamma^0_{-\ep_{2j-1}-\ep_{2j}}$ with $s/2+1\le j\le[(n-1)/2]$ since $\alpha\in\Delta^-_{\pi'_1}$ whilst $\beta\in\Delta^-_{\pi'_2}$ nor it is possible that
$\beta\in O^m$. 

\underline{Fifth case} : $\alpha\in\Gamma^0_{-\ep_{2i-1}-\ep_{2i}}\subset O^-$ with $s/2+1\le i\le [(n-1)/2]$. If $\beta\in\Gamma^0_{-\ep_{2j-1}-\ep_{2j}}$ with $s/2+1\le j\le [(n-1)/2]$ then one checks that $i=j$ and $\alpha+\beta=-\ep_{2i-1}-\ep_{2i}$, thus $\beta=\theta(\alpha)$. 

Assume now that $\beta\in\Gamma^0_{\ep_s-\ep_n}$. Then necessarily $\alpha=-\ep_{2i-1}+\ep_n$ and $\beta=\ep_s+\ep_{2i-1}$. Moreover $\gamma=\ep_s+\ep_{2i}\in\Gamma^0_{\ep_s+\ep_n}$ verifies $\theta(\alpha)+\gamma=\ep_s-\ep_n$ and $\theta(\beta)+\theta(\gamma)=-\ep_{2i-1}-\ep_{2i}$. All these roots belong to $O_2$. Hence $\alpha\in O_{cyc}$. A similar computation shows that  if $\beta\in\Gamma^0_{\ep_s+\ep_n}$, then $\alpha\in O_{cyc}$. \end{proof}

\begin{lemma}\label{Condbis D}
Let $\alpha\in O^m$. Then $\alpha\in O_{st}$ or $\alpha\in O_{cyc}$ or there exists $\beta\in O_3\cap O_{cyc}$ such that $\alpha=\tilde\beta$ or $\theta(\alpha)=\tilde\beta$.
\end{lemma}

\begin{proof}
Assume that $\alpha\in\Gamma^0_{\ep_s-\ep_n}\subset O^m$. 

\underline{ First case} : $\alpha=\ep_s-\ep_{2i-1}$ with $1\le i\le s/2-1$. Then there exists $\beta=\ep_{2i-1}+\ep_n\in O^+$ such that $\alpha+\beta\in S$ and first case of the proof of Lemma~\ref{Cond D}, $\beta\in O_{cyc}$ and $\alpha\in O_{cyc}$. If now $\alpha=\ep_{2i-1}-\ep_n$ with $1\le i\le s/2-1$ then $\alpha=\theta(\ep_s-\ep_{2i-1})\in O_{cyc}$ (by remark~\ref{remcyc} (1))
since $\ep_s-\ep_{2i-1}\in O_{cyc}$.

\underline{Second case} : $\alpha=\ep_s-\ep_{s-1}$.  By second case of the proof of Lemma~\ref{Cond D}, $\alpha\in O_{cyc}$. If now $\alpha=\ep_{s-1}-\ep_n$ then $\alpha=\theta(\ep_s-\ep_{s-1})$, thus $\alpha\in O_{cyc}$ by remark~\ref{remcyc} (1).

\underline{Third case} : $\alpha=\ep_s-\ep_{2i-1}$ with $s/2+2\le i\le [n/2]$. Then $\beta=\ep_{2i-1}+\ep_n\in\Gamma^0_{\ep_{2i-2}+\ep_{2i-1}}\subset O^+$ is such that (third case of the proof of Lemma~\ref{Cond D}) $\beta\in O_{cyc}$ and $\alpha\in C_{\beta}$, thus by remark~\ref{remcyc} (1), one has that $\alpha\in O_{cyc}$. If now $\alpha=\ep_{2i-1}-\ep_n$, with  $s/2+2\le i\le [n/2]$, then
$\alpha=\theta(\ep_s-\ep_{2i-1})$ and by remark~\ref{remcyc} (1), one has that $\alpha\in O_{cyc}$.

\underline{Fourth case} : $\alpha=\ep_s+\ep_{2i+1}$ with $s/2\le i\le[(n-2)/2]$. If $i\neq[(n-2)/2]$ or $n$ odd then $\beta=-\ep_{2i+1}+\ep_n\in\Gamma^0_{-\ep_{2i+1}-\ep_{2i+2}}\subset O^-$ is such that  (fifth case of the proof of Lemma~\ref{Cond D}), $\beta\in O_{cyc}$ and $\alpha\in C_{\beta}$ and then, by remark~\ref{remcyc} (1), $\alpha\in O_{cyc}$.

If $i=(n-2)/2$ and $n$ even, then $\alpha=\ep_s+\ep_{n-1}\in O_1$ and $\theta(\alpha)=-\ep_{n-1}-\ep_n\in O_1$ then in this case $\alpha\in O_{st}$.

Now suppose that $\alpha=-\ep_{2i+1}-\ep_n$ with $s/2\le i\le[(n-2)/2]$. If $i\neq[(n-2)/2]$ or $n$ odd, then $\alpha=\theta(\ep_s+\ep_{2i+1})$, hence by remark~\ref{remcyc} (1), $\alpha\in O_{cyc}$.

Finally if $\alpha=-\ep_{n-1}-\ep_n$ (that is, when $n$ is even) then $\alpha=\theta(\ep_s+\ep_{n-1})\in O_{st}$ (since $\theta(\alpha)\in O_{st}$).\par

Similar computations may be done for $\alpha\in\Gamma^0_{\ep_s+\ep_n}$; we leave the details as an exercise. We conclude that condition (4) of Lemma~\ref{non-degeneracy} holds.
\end{proof}

Let $T$ denote the complement of the set $\Gamma=\Gamma^+\sqcup\Gamma^-\sqcup\Gamma^m$ in $\Delta^+\sqcup\Delta^-_{\pi'}$.

\begin{lemma}\label{setTD}
One has that $\lvert T\rvert=\ind\p$.
\end{lemma}

\begin{proof}
One checks that $T=\{\ep_{s-1}+\ep_s,\,\ep_{s-1}-\ep_{s+1},\,\ep_{2i-1}-\ep_{2i},\,\ep_{2j}-\ep_{2j+1},\,-\ep_{2k+1}+\ep_{2k+2}\,\mid\, 1\le i\le s/2,\,s/2+1\le j\le [(n-1)/2],\,s/2\le k\le[(n-2)/2]\}$. Comparing with the proof of Lemma~\ref{setTB}, we see that it coincides with the set $T$ in type ${\rm B}_n$ for the same $s$. Hence $\lvert T\rvert=n-s/2+1$.

Moreover the $\langle {\bf ij}\rangle$-orbits are the same as in type ${\rm B}_n$ hence they are $n-s/2+1$ in number. Thus $\lvert T\rvert=\ind\p^-_{\pi',\, \Lambda}$.
\end{proof}

\begin{Remark} \rm
All conditions of Lemma~\ref{refineregular} are satisfied. Hence defining $h\in\h_{\Lambda}$ by $\gamma(h)=-1$ for all $\gamma\in S$,
and setting $y=\sum_{\gamma\in S}x_{\gamma}$ we obtain an adapted pair $(h,\,y)$ for $\p^-_{\pi',\,\Lambda}$.
\end{Remark}

 As in type ${\rm B}$, we give an expansion of the semisimple element $h$ and of its eigenvalues on the set $T$:

\begin{lemma}\label{eigenvalues D}
$$\begin{array}{lll}
h=\sum_{k=1}^{[s/4]}\bigl(\frac{s}{2}+2k-1\bigr)\ep_{2k-1}+\sum_{k=[s/4]+1}^{s/2-1}\bigl(\frac{3s}{2}-2k\bigr)\ep_{2k-1}\\
-\sum_{k=1}^{[s/4]}\bigl(\frac{s}{2}+2k\bigr)\ep_{2k}
-\sum_{k=[s/4]+1}^{s/2-1}\bigl(\frac{3s}{2}+1-2k\bigr)\ep_{2k}+\frac{s}{2}\ep_{s-1}-\ep_s\\
+\sum_{k=1}^{[(n-s)/2]}\bigl(-2k+1-\frac{s}{2}\bigr)\ep_{s+2k-1}+\sum_{k=1}^{[(n-s-1)/2]}\bigl(2k+\frac{s}{2}\bigr)\ep_{s+2k}.
\end{array}$$
The eigenvalues of $\ad h$ on $\g_T$ are:
\begin{itemize}
\item $s+4i-1=(\ep_{2i-1}-\ep_{2i})(h)$ for all $i$, with $ 1\le i\le [s/4]$.
\item $3s-4i+1=(\ep_{2i-1}-\ep_{2i})(h)$ for all $i$, with $[s/4+1]\le i\le s/2-1$.
\item $s/2+1=(\ep_{s-1}-\ep_s)(h)$.
\item $s/2-1=(\ep_{s-1}+\ep_s)(h)$.
\item $n-s/2-1=\begin{cases}
(-\ep_{n-1}+\ep_n)(h)\; \hbox{\rm {if $n$ even.}}\\
(\ep_{n-1}-\ep_n)(h)\; \hbox{\rm { if $n$ odd.}}
\end{cases}$
\item $s+1=(\ep_{s-1}-\ep_{s+1})(h)$.
\item $s+4j-1=(-\ep_{s+2j-1}+\ep_{s+2j})(h)$ for all $j$, with $1\le j\le[(n-s-1)/2]$.
\item $s+4j+1=(\ep_{s+2j}-\ep_{s+2j+1})(h)$ for all $j$, with $1\le j\le[(n-s-2)/2]$.
\end{itemize}
From the last three equalities we have that $s+2k-1$ is an eigenvalue of $\ad h$ on $\g_T$, for all $k$, with $1\le k\le n-s-1$.
\end{lemma}

Recall the bounds $\ch\mathcal A$ and $\ch\mathcal B$ for $\ch Y(\p)$  as well as the improved upper bound $\mathcal B'$  of Section~\ref{standardnotation}. We will show that the lower bound $\ch\mathcal A$ and the improved upper bound $\mathcal B'$ coincide, hence $Y(\p)$ is a polynomial algebra over $k$.

\begin{lemma}\label{calculboundsD}
The lower bound $\ch\mathcal A$ is equal to $$\prod_{\Gamma\in E(\pi')}(1-e^{\delta_{\Gamma}})^{-1}=(1-e^{-\varpi_s})^{-3}(1-e^{-2\varpi_s})^{-(n-2-s/2)}.\eqno (3)$$
\end{lemma}

\begin{proof}
The computation of the $\delta_\Gamma,\, \Gamma\in E(\pi')$, is exactly as in the proof of Lemma~\ref{calculboundsB}, except for the $\langle {\bf ij}\rangle$-orbit $\Gamma_{n-1}=\{\alpha_{n-1}\}$, for which $\delta_{\Gamma_{n-1}}=-2(\varpi_{n-1}-\varpi'_{n-1})=-\varpi_s$.
\end{proof}

We will now compute the improved upper bound $\mathcal B'$ for $\ch Y(\p)$ and as in type ${\rm B}_n$, we will show that it is equal to the lower bound.

\begin{lemma}\label{improvboundD}
The bound $\mathcal B'$ is given by the right hand side of (3). Hence one has that $\ch\mathcal A=\mathcal B'$ and then $Y(\p)$ is a polynomial algebra over $k$.
\end{lemma}

\begin{proof}
With the notations of Section~\ref{standardnotation} we have
$$\mathcal B'= \prod_{\gamma\in T}(1-e^{-(\gamma+t(\gamma))})^{-1},$$
where the set $T$ is given in the proof of Lemma~\ref{setTD}.

Again the computations are very similar to type ${\rm B}$ of Section~\ref{SecB}. If we compare the sets $S$ in Sections~\ref{SecB} and~\ref{SecD}) and the sets $T$ in the proofs of Lemmas~\ref{setTB} and~\ref{setTD} for type ${\rm B}_n$ and ${\rm D}_n$ with the same $s$, $2\le s\le n-2$, the sets $T$ are identical and the sets $S^\pm$ differ only by one element.

More precisely, if $n$ is odd, then $\ep_{n-1}+\ep_n\not\in S$, and so in this case for $\gamma=\ep_{n-1}-\ep_n\in T$, the element $t(\gamma)$ computed in the proof of Lemma~\ref{improvboundB} is no longer in $\mathbb Q S$. On the other hand,
$$\begin{array}{ll}
t(\ep_{n-1}-\ep_n)=\displaystyle(\ep_1+\ep_2)+\ldots+(\ep_{s-3}+\ep_{s-2})+(\ep_{s-1}+\ep_{s+1})\\
\displaystyle -\sum_{j=s/2+1}^{(n-1)/2}(\ep_{2j-1}+\ep_{2j})+\sum_{j=s/2+1}^{(n-3)/2}(\ep_{2j}+\ep_{2j+1})+(\ep_s+\ep_n)\in\mathbb Q S\end{array}$$
and $t(\ep_{n-1}-\ep_n)+(\ep_{n-1}-\ep_n)=\varpi_s$.

Similarly, if $n$ is even, then $-(\ep_{n-1}+\ep_n)\not\in S$ and for $\gamma=-\ep_{n-1}+\ep_n\in T$
one has that
$$\begin{array}{ll}t(-\ep_{n-1}+\ep_n)=\displaystyle(\ep_1+\ep_2)+\ldots+(\ep_{s-3}+\ep_{s-2})+(\ep_{s-1}+\ep_{s+1})\\
\displaystyle -\sum_{j=s/2+1}^{(n-2)/2}(\ep_{2j-1}+\ep_{2j})+\sum_{j=s/2+1}^{(n-2)/2}(\ep_{2j}+\ep_{2j+1})+(\ep_s-\ep_n)\in\mathbb Q S\end{array}$$
and $t(-\ep_{n-1}+\ep_n)+(-\ep_{n-1}+\ep_n)=\varpi_s$.

Note that in type ${\rm B}$ the corresponding weights $\gamma+t(\gamma)$ are equal to $2\varpi_s$ instead of $\varpi_s$, hence the improved upper bound for $Y(\p)$ in Lemma~\ref{improvboundB} will differ from the improved upper bound for ${\rm D}$ only by this factor. We conclude that the improved upper bound for $\ch Y(\p)$ is equal to the lower bound.
\end{proof}

Recall Lemma~\ref{improvboundD} and what we said at the end of Section~\ref{standardnotation}.
We then deduce the following theorem.

\begin{theorem}\label{conD}
Let $\g$ be a simple Lie algebra of type ${\rm D}_n$ ($n\ge 4$) and let $\p=\p^-_{\pi',\,\Lambda}$ be the  truncated maximal parabolic subalgebra of $\g$ associated to $\pi'=\pi\setminus\{\alpha_s\}$, where $s$ is an even integer with $s\le n-2$.

There exists an adapted pair $(h,\,y)$ for $\p$ and an affine slice $y+\g_T$ in $\p^*$ such that restriction of functions gives an isomorphism of algebras between $Y(\p)$ and the ring $R[y+\g_T]$ of polynomial functions on $y+\g_T$.

In particular $Y(\p)$ is a polynomial algebra over $k$
and the field $C(\p^-_{\pi'})$ of invariant fractions is a purely transcendental extension of $k$.
\end{theorem}

\begin{Remark} \rm
(1) When $s=2$, the above result was known in ~\cite{PPY} and was proven again by a different method in~\cite{J6bis}, where an adapted pair $(h',\, y'=\sum_{\gamma\in S'}x_{\gamma})$ was constructed. Our adapted pair $(h,\,y=\sum_{\gamma\in S}x_{\gamma})$ does not coincide with $(h',\,y')$ but it is equivalent to it. Indeed, setting $r_{i,\, j}=r_{\ep_i-\ep_j}\circ r_{\ep_i+\ep_j}$, one verifies that $w=\prod_{k=1}^{2m-1}r_{2k+1,\,2k+3}\circ r_{\alpha_1}$ $($resp. $w=\prod_{k=1}^{2m-3}r_{2k+1,\,2k+3}\circ r_{\alpha_1}\circ r_{n-1,\,n})$ if $n=4m+u$ with $u\in\{1,\,2,\,3\}$ $($resp. if $n=4m)$ and $m\neq 0$ is such that $w\in W_{\pi'}$ and sends bijectively $S$ to $S'$.\par
(2) The degrees of a set of homogeneous generators of $Y(\p)$ are the eigenvalues of $\ad h$ on $\g_T$ given in Lemma~\ref{eigenvalues D} each augmented by one.
\end{Remark}

\section{Another Lemma of non-degeneracy}

It remains to consider the case when the simple Lie algebra $\g$ is of type ${\rm D}_n$ and when the truncated maximal parabolic subalgebra $\p$ corresponds to $\pi'=\pi\setminus\{\alpha_n\}$ with $n$ even (such a case  will be called the extremal case). In this extremal case, the set $S$ and the Heisenberg sets $\Gamma_{\gamma}$, $\gamma\in S$, that we obtain (see next Section) will produce more roots in $O_3$ and  Lemma~\ref{non-degeneracy} will not apply.

Actually we will state a new Lemma of non-degeneracy, where we need  to extend the notions of stationary roots and of cyclic roots that we have defined in Definitions~\ref{defsr} and~\ref{cyc}.

Recall the hypotheses and notations of Section~\ref{sr}, especially a sequence $(\alpha^i)_{i\in\mathbb N}$ of roots in $O$ constructed from the root $\alpha\in O$ and the set
$A_{\alpha}=\{\alpha^i,\,\theta(\alpha^i)\mid i\in\mathbb N\}$.

We need the following condition:

{\bf Condition $(*)$}: If $\alpha\in O_3$, then there exists $\alpha'\in S_{\alpha}\setminus\{\theta(\alpha)\}\cap O_2$ such that $\theta(\alpha')\in O_1$. 

Assume that condition $(*)$ is satisfied for $\alpha\in O_3$ and choose a root $\alpha'$ as above. Then we define $\alpha^{(1)}$ as the unique root in $O$, distinct from $\alpha'$ and from $\theta(\alpha)$, such that $\alpha^{(1)}+\alpha\in S$. If $\theta(\alpha^{(1)})\in O_3$ satisfies condition $(*)$,  we define $\alpha^{(2)}$ similarly. If at each step $i$, condition $(*)$ is satisfied for the root $\theta(\alpha^{(i)})$ if it belongs to $O_3$ or if $\theta(\alpha^{(i)})\in O_1\sqcup O_2$, then  the sequence $(\alpha^{(i)})_{i\in\mathbb N}$ of roots in $O$ constructed from $\alpha^{(0)}=\theta(\alpha)$ is uniquely defined.

\begin{Remark}\label{est}\rm
Let $\alpha\in O$ such that $A_{\alpha}\cup A_{\theta(\alpha)}\subset O_1\sqcup O_2\sqcup O_3$, with condition $(*)$ satisfied for all roots in $(A_{\alpha}\cup A_{\theta(\alpha)})\cap O_3$. In particular this implies that the sequences $(\alpha^i)_{i\in\mathbb N}$ and $(\alpha^{(i)})_{i\in\mathbb N}$ are uniquely defined and moreover Remark~\ref{rem1}(1) still applies. Hence if there exists $i_0\in\mathbb N$ such that the sequence of roots in $O$ constructed from $\alpha^{i_0}$, resp. from $\alpha^{(i_0)}$, is stationary, then the sequence of roots in $O$ constructed from $\alpha$, resp. from $\theta(\alpha)$, is also stationary.
\end{Remark}

Here is the extension of Definition~\ref{defsr} of a stationary root:
\begin{Definition}\rm
Let $\alpha\in O$. We will say that $\alpha$ is an {\it extended stationary root} if $A_{\alpha}\cup A_{\theta(\alpha)}\subset O_1\sqcup O_2\sqcup O_3$, with condition $(*)$ satisfied for all roots 
in $(A_{\alpha}\cup A_{\theta(\alpha)})\cap O_3$ and if the sequences $(\alpha^{(i)})_{i\in\mathbb N}$ and $(\alpha^{i})_{i\in\mathbb N}$ are stationary. The set of extended stationary roots will be denoted by $O_{st}^e$.
\end{Definition}

\begin{Remark}\rm We have that $O_{st}\subset O_{st}^e$.\end{Remark}

Recall conditions (i)-(vi) of Definition~\ref{cyc} and its notations.
We replace condition (vi) by condition (vie) below:

{\bf Condition (vie)}: Let $\alpha\in O$ be a root verifying conditions (i)-(v) of Definition~\ref{cyc}. If $\delta\in C_{\alpha}\cap O_3$, then there exists $\tilde\delta\in S_{\delta}$ such that $A_{\tilde\delta}\subset O_1\sqcup O_2\sqcup O_3$ with condition $(*)$ satisfied for all roots in 
$A_{\tilde\delta}\cap O_3$, and the sequence $(\tilde\delta^i)_{i\in\mathbb N}$ is stationary.

\begin{Remark}\rm
1) With the above notations, the root $\tilde\delta$ is unique since $\tilde\delta\in S_{\delta}\setminus S_{\delta}\cap C_{\delta}$.

2) Condition (vi) of Definition~\ref{cyc} implies condition (vie), since $\theta(\tilde\delta)\in O_1$ implies that $\tilde\delta^1=\tilde\delta$.
\end{Remark}

Here is the extension of Definition~\ref{cyc} of a cyclic root:

\begin{Definition}\label{cyce}\rm
Let $\alpha\in O$. We say that $\alpha$ is an {\it extended cyclic root} if there exist $\beta,\,\gamma\in O$ satisfying conditions (i)-(v) of Definition~\ref{cyc} and condition (vie) above.
The set of extended cyclic roots will be denoted by $O_{cyc}^e$.
\end{Definition}

\begin{Remark}\rm We have that $O_{cyc}\subset O_{cyc}^e$.\end{Remark}

Similarly to Lemma~\ref{stat}, we obtain the following Lemma:

\begin{lemma}\label{state}
Let $\alpha\in O_{st}^e$. Let $\vartheta:O\to O$ be a permutation such that, for all $\gamma\in O$, $\gamma+\vartheta(\gamma)\in S$. Then the restriction of $\vartheta$ to $A_{\alpha}\cup A_{\theta(\alpha)}$ coincides with the involution $\theta$. Moreover the map $\vartheta$ exchanges $\beta'$ and $\theta(\beta')$, where $\beta'$ is the chosen root in $S_{\beta}\setminus\{\theta(\beta)\}\cap O_2$ such that $\theta(\beta')\in O_1$ for any $\beta\in (A_{\alpha}\cup A_{\theta(\alpha)})\cap O_3$.
\end{lemma}

\begin{proof}
It is similar to the proof of Lemma~\ref{stat}, noting that necessarily, with the above notations, the map $\vartheta$ sends $\theta(\beta')$ to $\beta'$ since $\theta(\beta')\in O_1$.
\end{proof}

Similarly to Lemma~\ref{permcyc}, we obtain the following Lemma:

\begin{lemma}\label{permcyce}
Let $\vartheta:O\to O$ be a permutation such that for all $\gamma\in O$, $\gamma+\vartheta(\gamma)\in S$. Then $\vartheta$ exchanges $\tilde\delta^i$ and $\theta(\tilde\delta^i)$, for all $i\in\mathbb N$, where $\tilde\delta$ is the root in $S_{\delta}$ given by condition (vie) of Definition~\ref{cyce}, for any $\delta\in O_{cyc}^e\cap O_3$.
\end{lemma}

%\begin{proof}
%It is similar to the above proofs.
%\end{proof}

Similarly to Lemma~\ref{non-degeneracy}, we get the following new Lemma of non-degeneracy:

\begin{lemma}\label{enewnon-degeneracy}

Assume  that:
\begin{enumerate}
\item $S_{\mid\h_{\Lambda}}$ is a basis for $\h_{\Lambda}^*$.
\item If $\alpha\in \Gamma_{\gamma}^0$, with $\gamma\in\ S^+$, then $S_{\alpha}\cap O^+=\{\theta(\alpha)\}$.

\item  If $\alpha\in \Gamma_{\gamma}^0$, with $\gamma\in\ S^-$, then $S_{\alpha}\cap O^-=\{\theta(\alpha)\}$.

 \item If $\alpha\in O$, with $S_{\alpha}\cap O^m\neq\emptyset$, then $\alpha\in O_{st}^e$ or $\alpha\in O_{cyc}^e$ or there exists $\beta\in O_{cyc}^e\cap O_3$ and $i\in\mathbb N$ such that $\alpha=\tilde\beta^i$ or $\theta(\alpha)=\tilde\beta^i$, where $\tilde\beta$ is the unique root in $S_{\beta}$ given by condition (vie) of Definition~\ref{cyce}.
 \end{enumerate}
Then the restriction of the bilinear form $\Phi_y$ to $\o\times\o$ is non-degenerate.
\end{lemma}

\begin{proof} It follows like Lemma~\ref{non-degeneracy}, using Lemmas~\ref{state} and~\ref{permcyce}.

\end{proof}

\section{Type D, the extremal case}\label{SecDe}

In this section, we assume that the simple Lie algebra $\g$ is of type ${\rm D}_n$ with $n\ge 6$ and $n$ even and we consider $\p=\p^-_{\pi',\,\Lambda}$ the  truncated maximal parabolic
subalgebra of $\g$ associated to $\pi'=\pi\setminus\{\alpha_n\}$. Then the lower and the upper bounds of Section~\ref{standardnotation} for $\ch Y(\p)$ do not coincide. We will construct an adapted pair for $\p$ and then prove that $Y(\p)$ is a polynomial algebra over $k$. Since the case when $\pi'=\pi\setminus\{\alpha_{n-1}\}$ is symmetric, this will also prove that $Y(\p)$
is polynomial when $\pi'=\pi\setminus\{\alpha_{n-1}\}$.

Observe that the polynomiality of $Y(\p)$ was already known by \cite[Thm. 2.3]{PPY3}, since $\p$ is the semi-direct product of its Levi factor $\g'\simeq\s\l_n$ and its nilradical $\m$, which is in this case an abelian ideal of $\g'$, isomorphic to $\Lambda^{2} k^n$ as a $\s\l_n$-module.
However the degrees of a set of homogeneous generators were not known.
In our present work we will also compute their degrees (see Lemma~\ref{eigenvaluesDe} and Thm~\ref{conDe}).

%\subsection{The set $S$.}\label{SDe}~~

We set $S=\{\ep_{2i-1}+\ep_{2i},\,\ep_{n-3}+\ep_{n-1},\,\ep_n-\ep_{n-3},\,\ep_{n-2}-\ep_{n-4},\,\ep_{n-4}-\ep_{n-5},\,\ep_{n-3}-\ep_{n-6},\,\ep_{n-2j}-\ep_{n-2j-2}\mid 1\le i\le n/2-2,\,3\le j\le n/2-2\}$. One checks that $S\subset\Delta^+\sqcup\Delta^-_{\pi'}$ and that $\lvert S\rvert =n-1=\dim\h_{\Lambda}$.

We first prove below that condition (1) of Lemma~\ref{enewnon-degeneracy} holds.

\begin{lemma}
$S_{\mid\h_{\Lambda}}$ is a basis for $\h_{\Lambda}^*$.
\end{lemma}

\begin{proof}
Set $S=\{s_i\}_{1\le i\le n-1}$ with $s_i=\ep_{2i-1}+\ep_{2i}$ for all $i\in\mathbb N$, $1\le i\le n/2-2$, $s_{n/2-1}=\ep_{n-3}+\ep_{n-1}$, $s_{n/2}=\ep_{n-4}-\ep_{n-5}$, $s_{n/2+1}=\ep_{n-2}-\ep_{n-4}$,
$s_{n/2+2}=\ep_n-\ep_{n-3}$, $s_{n/2+3}=\ep_{n-3}-\ep_{n-6}$, $s_{n/2+k}=\ep_{n-2k+2}-\ep_{n-2k}$ for all $k\in\mathbb N$, with $4\le k\le n/2-1$.
Then set $s'_i=s_i$ for all $i\in\mathbb N$ with $1\le i\le  n-1$, $i\neq n/2-1$, $s'_{n/2-1}=\ep_{n-1}+\ep_n=s_{n/2-1}+s_{n/2+2}$ and $S'=\{s'_i\}_{1\le i\le n-1}$.

If we choose $\{h_i\}_{1\le i\le n-1}=\{\alpha_i^\vee\}_{1\le i\le n-1}$ as a basis of $\h_{\Lambda}$, it is sufficient to show that $\det(s'_i(h_j))_{1\le i,\,j\le n-1}\neq 0$.

By ordering $S'$ as above and the basis of $\h_{\Lambda}$ as
$\{\alpha_{2i}^\vee,\,\alpha_{n-5}^\vee,\,\alpha_{n-3}^\vee,\,\alpha_{n-1}^\vee,\,\alpha_{n-2j-1}^\vee\mid 1\le i \le n/2-1,\,3\le j\le n/2-1\}$,
one checks that the matrix $(s'_i(h_j))_{1\le i,,j\le n-1}$ is a lower triangular matrix with $1$ on the first $n/2-2$ diagonal elements, then $-1$, $-2$,
$-1$, $-1$ on the next diagonal elements and then $1$ on the $n/2-3$ last diagonal elements. Hence $\det(s'_i(h_j))_{1\le i,,j\le n-1}=2$ and the lemma.
\end{proof}

%\subsection{The Heisenberg sets.} 
Now we define the Heisenberg sets $\Gamma_{\gamma}$ of centre $\gamma$, for all $\gamma\in S$.

\begin{itemize}
\item For all $k\in\mathbb N$, $2\le k\le n/2-3$,
 set $\Gamma_{\ep_{2k}-\ep_{2k-2}}=\{\ep_{2k}-\ep_{2k-2},\,\ep_{2k}-\ep_{i},\,\ep_i-\ep_{2k-2}\mid 1\le i\le 2k-3\}$.

\item Set $\Gamma_{\ep_{n-3}-\ep_{n-6}}=\{\ep_{n-3}-\ep_{n-6},\,\ep_{n-3}-\ep_i,\,\ep_i-\ep_{n-6}\mid 1\le i\le n-7\}$.
\item Set $\Gamma_{\ep_{n-4}-\ep_{n-5}}=\{\ep_{n-4}-\ep_{n-5},\,\ep_{n-3}-\ep_{n-5},\,\ep_{n-4}-\ep_{n-3},\,\ep_{n-4}-\ep_{2i},\,\ep_{2i}-\ep_{n-5}\mid 1\le i\le n/2-3\}$.
\item Set $\Gamma_{\ep_{n-2}-\ep_{n-4}}=\{\ep_{n-2}-\ep_{n-4},\,\ep_{n-2}-\ep_{n-1},\,\ep_{n-1}-\ep_{n-4},\,\ep_{n-2}-\ep_n,\,\ep_n-\ep_{n-4},\,\ep_{n-2}-\ep_i,\,\ep_i-\ep_{n-4}\mid 1\le i\le n-5\}$.
\item Set $\Gamma_{\ep_n-\ep_{n-3}}=\{\ep_n-\ep_{n-3},\,\ep_n-\ep_{n-2},\,\ep_{n-2}-\ep_{n-3},\,\ep_n-\ep_{n-1},\,\ep_{n-1}-\ep_{n-3},\,\ep_n-\ep_i,\,\ep_i-\ep_{n-3}\mid 1\le i\le n-6\}$.
\item Set $\Gamma_{\ep_{n-3}+\ep_{n-1}}=\{\ep_{n-3}+\ep_{n-1},\,\ep_{n-3}+\ep_n,\,\ep_{n-1}-\ep_n,\,\ep_{n-3}-\ep_n,\,\ep_n+\ep_{n-1},\,\ep_{n-3}-\ep_{n-2},\,\ep_{n-2}+\ep_{n-1},\,\ep_{n-3}+\ep_{n-2},\,-\ep_{n-2}+\ep_{n-1},\,\ep_{n-1}-\ep_i,\,\ep_i+\ep_{n-3}\mid 1\le i\le n-5\}$.
\end{itemize}

It is easy to check that the $n/2+1$ sets defined above are Heisenberg sets, which we denote by $\Gamma_{\gamma_j}$, $1\le j\le n/2+1$, whose centre will be denoted
by $\gamma_j\in S$.

Let $i\in\mathbb N$, with $1\le i\le n/2-2$, and recall that $\beta_i=\ep_{2i-1}+\ep_{2i}$ is an element of the Kostant cascade of $\g$ (see Section~\ref{SecD}) and that we denote by $H_{\beta_i}$ the maximal Heisenberg set in $\Delta^+$ of centre $\beta_i$ (see Example~\ref{MaxHeisenbergset}).

We define below every Heisenberg set $\Gamma_{\beta_i}$ of centre $\beta_i$, $1\le i\le n/2-2$, by decreasing induction on $i$.

\begin{itemize}
\item First we set $\Gamma_{\beta_{n/2-2}}=(H_{\beta_{n/2-2}}\setminus\bigsqcup_{1\le j\le n/2+1}\Gamma_{\gamma_j}\cap H_{\beta_{n/2-2}})\sqcup\{\ep_i+\ep_{n-4},\,\ep_{n-5}-\ep_i\mid 1\le i\le n-6\}\sqcup\{\ep_{n-4}-\ep_{2i-1},\,\ep_{2i-1}+\ep_{n-5}\mid 1\le i\le n/2-3\}$.
Set $\gamma_j=\beta_{n-j}$ for all $j\in\mathbb N$, with $n/2+2\le j\le n-1$
and suppose, for $2\le k\le n/2-2$, that we have defined the Heisenberg set $\Gamma_{\gamma_j}$ of centre $\gamma_j\in S$, for all $j\in\mathbb N$, with $1\le j\le n/2+k$.

\item Then we set $\Gamma_{\gamma_{n/2+k+1}}=\Gamma_{\beta_{n/2-k-1}}=(H_{\beta_{n/2-k-1}}\setminus\bigsqcup_{1\le j\le n/2+ k}\Gamma_{\gamma_j}\cap H_{\beta_{n/2-k-1}})\sqcup\{\ep_i+\ep_{n-2k-2},\,\ep_{n-2k-3}-\ep_i\mid 1\le i\le n-2k-4\}$.
\end{itemize}

One checks that, for every $\gamma_j\in S$,  $1\le j\le n-1$, the set $\Gamma_{\gamma_j}$ is a Heisenberg set of centre $\gamma_j$. Moreover by construction all these Heisenberg sets are disjoint and $\Gamma=\bigsqcup_{\gamma\in S}\Gamma_{\gamma}\subset\Delta^+\sqcup\Delta^-_{\pi'}$.

%\subsection{Conditions $(2)$, $(3)$ and $(4)$ of lemma~\ref{non-degeneracy}.}\label{conditionsED}~~

Recall the definition  of $\Gamma^\pm$ and of $\Gamma^m$ (Section~\ref{sr}) and
observe that, for $n\ge 8$, $\Gamma_{\beta_1}=\Gamma^+$ (that is, $\{\beta_1=\ep_1+\ep_2\}=S^+$) and $\Gamma^m=\bigsqcup_{\gamma\in S\setminus\{\beta_1\}}\Gamma_{\gamma}$ (that is, $S\setminus\{\beta_1\}=S^m$) and,
for $n=6$, $\Gamma_{\beta_1}= \Gamma^+$, $\Gamma_{\ep_6-\ep_3}=\Gamma^-$ (here, $\{\ep_6-\ep_3\}= S^-$) and $\Gamma^m=\bigsqcup_{\gamma\in S\setminus\{\beta_1,\,\ep_6-\ep_3\}}\Gamma_{\gamma}$ (here, $S\setminus\{\beta_1,\,\ep_6-\ep_3\}=S^m$).

In the extremal case, we will see that  conditions (2)-(3)-(4) of Lemma~\ref{enewnon-degeneracy} are more complicated to check than in the non-extremal case in type D since there are more roots in $O$ which donnot belong to $O_1\sqcup O_2$. We will show below that conditions (2)-(3)-(4) of Lemma~\ref{enewnon-degeneracy} hold.

\begin{lemma}\label{condDe}
Conditions (2)-(3)-(4) of Lemma~\ref{enewnon-degeneracy} hold.
\end{lemma}
\begin{proof}
(a) \underline{The roots in $\Gamma^0_{\ep_{2k}-\ep_{2k-2}}$, $2\le k\le n/2-3$}.

Here one checks that all roots belong to $O_{st}^e$.
Let us explain the case when $\alpha=\ep_{2k}-\ep_i$, $1\le i\le 2k-3$, with $i$ even.
If $i=2$, then $\theta(\alpha)\in O_2$, $\alpha^1=\ep_1+\ep_{2k-2}$ and $\theta(\alpha^1)\in O_1$, hence $\alpha^2=\alpha^1$. If $2k=n-6$, then $\alpha\in O_2$, $\alpha^{(1)}=\ep_{n-7}+\ep_i$ and $\theta(\alpha^{(1)})\in O_1$, hence $\alpha^{(2)}=\alpha^{(1)}$.
In the other cases, $\alpha\in O_3$ and $\theta(\alpha)\in O_3$ and they verify condition $(*)$. Indeed $\alpha'=\ep_{2k-1}+\ep_i\in\Gamma^0_{\beta_{i/2}}\cap O_2$ and $\theta(\alpha')=\ep_{i-1}-\ep_{2k-1}\in O_1$. Similarly $\theta(\alpha)'=\ep_{i-1}+\ep_{2k-2}\in\Gamma^0_{\beta_{k-1}}\cap O_2$ and $\theta(\theta(\alpha)')=\ep_{2k-3}-\ep_{i-1}\in O_1$.
Moreover $\alpha^{(1)}=\ep_{i+2}-\ep_{2k}\in\Gamma^0_{\ep_{2k+2}-\ep_{2k}}$ and $\alpha^1=\ep_{2k-2}-\ep_{i-2}\in\Gamma^0_{\ep_{2k-2}-\ep_{2k-4}}$,  then one deduces by induction that the sequences $(\alpha^{(i)})_{i\in\mathbb N}$ and  $(\alpha^{i})_{i\in\mathbb N}$ are stationary and that $A_{\alpha}\cup A_{\theta(\alpha)}\subset O_1\sqcup O_2\sqcup O_3$ with condition $(*)$ satisfied for all roots in $(A_{\alpha}\cup A_{\theta(\alpha)})\cap O_3$.

(b)\underline{ The roots in $\Gamma^0_{\ep_{n-3}-\ep_{n-6}}$}.

Here one checks that all roots belong to $O_{st}^e$. Let us explain the case when $\alpha=\ep_{n-3}-\ep_i$, $1\le i\le n-7$, with $i$ even. Then $\alpha$ and $\theta(\alpha)$ belong to $O_3$ and verify condition $(*)$. Indeed $\alpha'=\ep_{n-1}+\ep_i\in\Gamma^0_{\beta_{i/2}}\cap O_2$, $\theta(\alpha')=\ep_{i-1}-\ep_{n-1}\in O_1$, and $\theta(\alpha)'=\ep_{i-1}+\ep_{n-6}\in\Gamma^0_{\beta_{n/2-3}}\cap O_2$ and $\theta(\theta(\alpha)')=\ep_{n-7}-\ep_{i-1}\in O_1$. Moreover $\alpha^1=\ep_{n-6}-\ep_{i-2}\in\Gamma^0_{\ep_{n-6}-\ep_{n-8}}$, and  paragraph $(a)$ above gives that $\alpha^1\in O_{st}^e$ then, by Remark~\ref{est},  the sequence $(\alpha^{i})_{i\in\mathbb N}$ is stationary and $A_{\alpha}\subset O_1\sqcup O_2\sqcup O_3$ with condition $(*)$ satisfied for all roots in $A_{\alpha}\cap O_3$.
On the other hand, one has that $\alpha^{(1)}=\ep_{i+2}-\ep_{n-3}\in\Gamma^0_{\ep_n-\ep_{n-3}}\cap O_2$, $\theta(\alpha^{(1)})=\ep_n-\ep_{i+2}\in O_2$, (unless  $i=n-8$, in which case $\theta(\alpha^{(1)})\in O_1$), $\alpha^{(2)}=\ep_{i+4}-\ep_n\in\Gamma^0_{\beta_{(i+4)/2}}\cap O_2$ and $\theta(\alpha^{(2)})=\ep_{i+3}+\ep_n\in O_1$. Hence $\alpha^{(3)}=\alpha^{(2)}$ and the sequence $(\alpha^{(i)})_{i\in\mathbb N}$ is stationary and $A_{\theta(\alpha)}\subset O_1\sqcup O_2\sqcup O_3$ with condition $(*)$ satisfied for all roots in $A_{\theta(\alpha)}\cap O_3$. This proves that $\alpha\in O_{st}^e$.

(c) \underline{The roots in $\Gamma^0_{\ep_{n-4}-\ep_{n-5}}$}.

Here one checks that all roots belong to $O_{cyc}^e$.

Let us explain the case when $\alpha=\ep_{n-4}-\ep_{2i}\in\Gamma^0_{\ep_{n-4}-\ep_{n-5}}$, $1\le i\le n/2-3$. Let $\beta=\ep_{2i-1}-\ep_{n-5}\in\Gamma^0_{\beta_i}$ and $\gamma=\ep_{2i-1}+\ep_{n-5}\in\Gamma^0_{\ep_{n-4}+\ep_{n-5}}$. Then $\alpha,\,\beta,\,\gamma$ verify  the cyclic relations (i)-(iii) of Definition~\ref{cyce} and $\beta,\,\theta(\beta),\,\gamma,\,\theta(\gamma)$ belong to $O_2$. Moreover $\alpha$ and $\theta(\alpha)$ belong to $O_3$ (unless $2i= n-6$, in which case $\alpha\in O_2$ or $i=1$, in which case $\theta(\alpha)\in O_2$). If $2i\le n-8$, $\tilde\alpha=\ep_{2i+2}-\ep_{n-4}\in\Gamma^0_{\ep_{n-2}-\ep_{n-4}}\cap O_3\cap S_{\alpha}$  is such that $\theta(\tilde\alpha)=\ep_{n-2}-\ep_{2i+2}\in O_2$ and, if $2i\le n-10$, $\tilde\alpha^1=\ep_{2i+4}-\ep_{n-2}\in\Gamma^0_{\beta_{i+2}}\cap O_2$ and $\theta(\tilde\alpha^1)=\ep_{2i+3}+\ep_{n-2}\in O_1$ and if $2i=n-8$, then $\tilde\alpha^1=\ep_{n-3}-\ep_{n-2}\in\Gamma^0_{\ep_{n-3}+\ep_{n-1}}\cap O_2$ and $\theta(\tilde\alpha^1)=\ep_{n-2}+\ep_{n-1}\in O_1$.

Let $\tilde\alpha'=\ep_{2i+1}+\ep_{n-4}\in\Gamma^0_{\beta_{n/2-2}}\cap O_2\cap S_{\tilde\alpha}$. Then $\theta(\tilde\alpha')=\ep_{n-5}-\ep_{2i+1}\in O_1$. Hence $\tilde\alpha$ satisfies  condition $(*)$ and the sequence $(\tilde\alpha^i)_{i\in\mathbb N}$ is stationary.

Similarly if $i\ge 2$, one verifies that the sequence $(\widetilde{\theta(\alpha)}^i)_{i\in\mathbb N}$ is stationary and that $A_{\widetilde{\theta(\alpha)}}\subset O_1\sqcup O_2$.

(d) \underline{The roots in $\Gamma^0_{\ep_{n-2}-\ep_{n-4}}$}.
Here one checks that there exists $\beta\in O_3\cap O_{cyc}^e$ and $i\in\mathbb N$ such that 
$\alpha=\tilde\beta^i$ or $\theta(\alpha)=\tilde\beta^i$, unless some particular cases for which $\alpha\in O_{st}$.

For instance assume that $\alpha=\ep_{n-2}-\ep_i\in\Gamma^0_{\ep_{n-2}-\ep_{n-4}}$, $1\le i\le n-7$ with
$i$ is odd. Let $\beta=\ep_{i+3}-\ep_{n-5}$ if $i\le n-9$, resp. $\beta=\ep_{n-3}-\ep_{n-5}$ if $i=n-7$.
By paragraph (c) above one has that $\beta\in\Gamma^0_{\ep_{n-4}-\ep_{n-5}}\cap O_3\cap O_{cyc}^e$.Then $\tilde\beta=\ep_{n-5}-\ep_{i+1}\in\Gamma^0_{\beta_{n/2-2}}\cap S_{\beta}\cap O_2$, $\theta(\tilde\beta)=\ep_{n-4}+\ep_{i+1}\in O_2$ and $\tilde\beta^1=\ep_i-\ep_{n-4}\in\Gamma^0_{\ep_{n-2}-\ep_{n-4}}\cap O_2$ and $\alpha=\theta(\tilde\beta^1)\in O_1$. Hence $\beta$ satisfies condition (vie) of Definition~\ref{cyce} and $\theta(\alpha)={\tilde\beta}^1$.

Hence $\alpha$ satisfies the last part of condition (4) of Lemma~\ref{enewnon-degeneracy}.

(e) \underline{The roots in $\Gamma^0_{\ep_n-\ep_{n-3}}$}.

Here one checks that all roots belong to $O_{st}^e$.

(f) \underline{The roots in $\Gamma^0_{\ep_{n-3}+\ep_{n-1}}$}.

Here one checks that all roots belong to $O_{st}^e$, except when $\alpha=\ep_{n-3}-\ep_{n-2}$ with $n\ge 10$, in which case $\alpha=\tilde\beta^1$, for $\beta=\ep_{n-4}-\ep_{n-8}\in\Gamma^0_{\ep_{n-4}-\ep_{n-5}}\cap O_3\cap O_{cyc}^e$  by paragraph (c) above; or when $\alpha=\ep_{n-1}-\ep_{n-5}$, or $\alpha=\ep_{n-5}+\ep_{n-3}$, in which case $\alpha\in O_{cyc}^e$ by paragraph (c) above.

(g) \underline{The roots in $\Gamma^0_{\beta_{n/2-2}}$}.

One has that $\beta_{n/2-2}=\ep_{n-5}+\ep_{n-4}$ and
$\Gamma^0_{\beta_{n/2-2}}=\{\ep_{n-5}+\ep_i,\,\ep_{n-4}-\ep_i,\,\ep_{n-5}-\ep_j,\,\ep_j+\ep_{n-4},\,\ep_{n-5}+\ep_{2k-1},\,\ep_{n-4}-\ep_{2k-1}\mid
n-2\le i\le n,\, 1\le j\le n,\,j\not\in\{n-5,\,n-4\},\,1\le k\le n/2-3\}$.

Using paragraphs (c) or (d) above, one checks that $\alpha\in O_{st}^e$ or that $\alpha\in O_{cyc}^e$, unless $\alpha=\ep_{n-5}-\ep_j$,  or $\alpha=\ep_j+\ep_{n-4}$, $1\le j\le n-6$, $j$ even, in which case $\alpha=\tilde\beta$ or $\theta(\alpha)=\tilde\beta$, where $\beta\in\Gamma^0_{\ep_{n-4}-\ep_{n-5}}\cap O_{cyc}^e\cap O_3$ by paragraph (c) above.

Thus condition (4) of Lemma~\ref{enewnon-degeneracy} holds for all roots in $\Gamma^0_{\beta_{n/2-2}}$.
Also condition (2) for $n=6$ holds since one may verify that, if $\alpha\in\Gamma^0_{\beta_1}$, then $S_{\alpha}\cap\Gamma^0_{\beta_1}=\{\theta(\alpha)\}$.

(h) \underline{The roots in $\Gamma^0_{\beta_{i}}$, with $1\le i\le n/2-3$}.

Observe that this implies that $n\ge 8$.

Recall that $\beta_i=\ep_{2i-1}+\ep_{2i}$ and observe that
$\Gamma^0_{\beta_i}=\{\ep_{2i-1}+\ep_{2j-1},\,\ep_{2i}-\ep_{2j-1},\,\ep_{2i-1}-\ep_{2k-1},\,\ep_{2i}+\ep_{2k-1},\,\ep_{2i-1}\pm\ep_u,\,\ep_{2i}\mp\ep_u\mid i+1\le j\le n/2-3,\,i+1\le k\le n/2-2,\,n-2\le u\le n\}\sqcup\{\ep_{2i-1}-\ep_v,\,\ep_{2i}+\ep_v\mid 1\le v\le 2i-2\}$.

Here one checks, using the above paragraphs, that $\alpha\in O_{st}^e$, unless $\alpha=\ep_{2i-1}-\ep_{n-5}\in O_{cyc}^e$, resp. $\alpha=\ep_{2i}+\ep_{n-5}\in O_{cyc}^e$, or when $\alpha=\ep_{2i-1}+\ep_{n-2}$, resp. $\alpha=\ep_{2i}-\ep_{n-2}$, and $i\ge 3$, in which case $\theta(\alpha)=\tilde\beta^1$, resp. $\alpha=\tilde\beta^1$, with $\beta\in\Gamma^0_{\ep_{n-4}-\ep_{n-5}}\cap O_{cyc}^e$ by paragraph (c) above.

Thus condition (4) of Lemma~\ref{enewnon-degeneracy} holds for all roots in $\Gamma^0_{\beta_{i}}$, with $1\le i\le n/2-3$. Also condition (2) for $i=1$ and $n\ge 8$ holds since  one may verify that, if $\alpha\in\Gamma^0_{\beta_1}$, then $S_{\alpha}\cap\Gamma^0_{\beta_1}=\{\theta(\alpha)\}$.
\end{proof}

%\subsection{The set $T$.}

Recall that we denote by $T$ the complement of the set $\Gamma=\bigsqcup_{\gamma\in S}\Gamma_{\gamma}$ in $\Delta^+\sqcup\Delta^-_{\pi'}$.

\begin{lemma}\label{TDe} 

One has that $\lvert T\rvert=\ind \p$.
\end{lemma}

\begin{proof}
One checks that $T=\{\ep_{n-3}-\ep_{n-1},\,\ep_{n-2}+\ep_n,\,\ep_n-\ep_{n-5},\,\ep_{n-3}-\ep_{n-4},\,\ep_{n-2k}-\ep_{n-2k-1}\mid 3\le k\le n/2-1\}$.
Then $\lvert T\rvert=n/2+1$.

Moreover the $\langle\bf{ ij}\rangle$-orbits in $\pi$ are $\Gamma_t=\{\alpha_t,\,\alpha_{n-t}\}$ for all $1\le t\le n/2-1$, $\Gamma_{n/2}=\{\alpha_{n/2}\}$ and $\Gamma_n=\{\alpha_n\}$.
They are $n/2+1$ in number, hence the lemma.
\end{proof}

\begin{Remark}\label{remcond}\rm
All conditions of Lemma~\ref{refineregular} are satisfied. Hence defining $h\in\h_{\Lambda}$ by $\gamma(h)=-1$ for all $\gamma\in S$,
and setting $y=\sum_{\gamma\in S}x_{\gamma}$ we obtain an adapted pair $(h,\,y)$  for $\p^-_{\pi',\,\Lambda}$.
\end{Remark}

%\subsection{The eigenvalues of $\ad h$ on $\g_T$.}

\begin{lemma}\label{eigenvaluesDe}
The semisimple element $h$ of the above adapted pair for $\p^-_{\pi',\,\Lambda}$ is

$h=-\ep_2+5\ep_3-2\ep_4-6\ep_5+4\ep_6$ for $n=6$, 
and for $n\ge 8$

\noindent
$\begin{array}{ll}
h=& -n\ep_1+\sum_{k=1}^{n/2-4}(k-n)\ep_{2k+1}+\sum_{k=1}^{n/2-3}(n-k)\ep_{2k}-\ep_{n-4}+
 (n/2+2)\ep_{n-3}\\&-2\ep_{n-2}-(n/2+3)\ep_{n-1}+(n/2+1)\ep_n.\end{array}$ 

Then the eigenvalues of $\ad h$ on $\g_T$ are :

$\bullet$ $2(n-i)+1=(\ep_{2i}-\ep_{2i-1})(h)$ for all $1\le i\le n/2-3$.

$\bullet$ $n+5=(\ep_{n-3}-\ep_{n-1})(h)$.

$\bullet$ $n/2-1=(\ep_{n-2}+\ep_n)(h)$.

$\bullet$ $n/2+1=(\ep_n-\ep_{n-5})(h)$.

$\bullet$ $n/2+3=(\ep_{n-3}-\ep_{n-4})(h)$.

From the first two equalities, we have that $n+4+2k-1$ is an eigenvalue of $\ad h$ on $\g_T$, for all $k\in\mathbb N$, with $1\le k\le n/2-2$.

\end{lemma}

%\subsection{Polynomiality of $Y(\p)$.}~~

%\subsubsection{The lower bound for $\ch Y(\p)$.}

\begin{lemma}\label{lbDe}
 The lower bound $\ch\mathcal A$ for $Y(\p)$ is equal to
$$\prod_{\Gamma\in E(\pi')}(1-e^{\delta_{\Gamma}})^{-1}=(1-e^{-2\varpi_n})^{-3}(1-e^{-4\varpi_n})^{-(n/2-2)}\eqno (4)$$
\end{lemma}
\begin{proof}
One checks that, for all $t\in\mathbb N$, $1\le t\le n-2$, $\varpi_t-\varpi'_t=(2t/n)\varpi_n$ and that $\varpi_{n-1}-\varpi'_{n-1}=((n-2)/n)\varpi_n$. Then with the notations of the proofs of Lemma~\ref{calculboundsB} and of Lemma~\ref{TDe} one has
that, for all $t\in\mathbb N$, $2\le t\le n/2-1$, $\delta_{\Gamma_t}=-4\varpi_n$, whereas $\delta_{\Gamma_1}=\delta_{\Gamma_{n/2}}=\delta_{\Gamma_n}=-2\varpi_n$. Hence equality (4) holds.
\end{proof}

%\subsubsection{The improved upper bound for $\ch Y(\p)$.}

 \begin{lemma}\label{improvboundDe}

The improved upper bound $\mathcal B'$ for ${\rm ch}\,Y(\p)$ is equal to the lower bound, given by the right hand side of (4).
\end{lemma}
\begin{proof}
Recall that, if for every $\gamma\in T$, $t(\gamma)$ denotes the unique element in $\mathbb QS$ such that $\gamma+t(\gamma)$ is a multiple of $\varpi_n$, then the improved upper bound is given by the product $\prod_{\gamma\in T}(1-e^{-(\gamma+t(\gamma))})^{-1}$.

One verifies that :

$\bullet$ $t(\ep_{n-2}+\ep_n)=\sum_{1\le i\le n/2-2}(\ep_{2i-1}+\ep_{2i})+(\ep_{n-3}+\ep_{n-1})$
and that $\ep_{n-2}+\ep_n+t(\ep_{n-2}+\ep_n)=2\varpi_n$.

$\bullet$ $t(\ep_n-\ep_{n-5})=\sum_{1\le i\le n/2-3}(\ep_{2i-1}+\ep_{2i})+2(\ep_{n-5}+\ep_{n-4})+(\ep_{n-3}+\ep_{n-1})+(\ep_{n-2}-\ep_{n-4})$ and $\ep_n-\ep_{n-5}+t(\ep_n-\ep_{n-5})=2\varpi_n$.

$\bullet$ $t(\ep_{n-3}-\ep_{n-4})=\sum_{1\le i\le n/2-3}(\ep_{2i-1}+\ep_{2i})+2(\ep_{n-5}+\ep_{n-4})+(\ep_{n-3}+\ep_{n-1})+(\ep_n-\ep_{n-3})+(\ep_{n-2}-\ep_{n-4})+(\ep_{n-4}-\ep_{n-5})$, and $\ep_{n-3}-\ep_{n-4}+t(\ep_{n-3}-\ep_{n-4})=2\varpi_n$.

$\bullet$ $t(\ep_{n-3}-\ep_{n-1})=2\sum_{1\le i\le n/2-3}(\ep_{2i-1}+\ep_{2i})+3(\ep_{n-5}+\ep_{n-4})+3(\ep_{n-3}+\ep_{n-1})+2(\ep_n-\ep_{n-3})+2(\ep_{n-2}-\ep_{n-4})+(\ep_{n-4}-\ep_{n-5})$ and $\ep_{n-3}-\ep_{n-1}+t(\ep_{n-3}-\ep_{n-1})=4\varpi_n$.

$\bullet$ For $3\le k\le n/2-1$, $t(\ep_{n-2k}-\ep_{n-2k-1})=2\sum_{1\le i\le n/2-3,\,i\neq n/2-k}(\ep_{2i-1}+\ep_{2i})+3(\ep_{n-5}+\ep_{n-4})+3(\ep_{n-2k-1}+\ep_{n-2k})+2(\ep_{n-3}+\ep_{n-1})+2(\ep_n-\ep_{n-3})+2(\ep_{n-2}-\ep_{n-4})+(\ep_{n-4}-\ep_{n-5})+2(\ep_{n-3}-\ep_{n-6})+2\sum_{3\le j\le k-1}(\ep_{n-2j}-\ep_{n-2j-2})$ and $\ep_{n-2k}-\ep_{n-2k-1}+t(\ep_{n-2k}-\ep_{n-2k-1})=4\varpi_n$.

Thus the improved upper bound is equal to the right hand side of $(4)$.
\end{proof}

%\subsubsection{Conclusion.}
%Recall~\ref{improvedupperbound} and~\ref{degrees}.
One can now give the following

\begin{theorem}\label{conDe}
Let $\g$ be a simple Lie algebra of type ${\rm D}_n$ with $n$ an even integer, $n\ge 6$, and let $\p=\p^-_{\pi',\,\Lambda}$ be the truncated maximal parabolic subalgebra of $\g$ associated to $\pi'=\pi\setminus\{\alpha_n\}$.

There exists an adapted pair $(h,\,y)$ for $\p$ and an affine slice $y+\g_T$ in $\p^*$ such that  restriction of functions gives an isomorphism of algebras
between $Y(\p)$ and the ring $R[y+\g_T]$ of polynomial functions on $y+\g_T$.

In particular $Y(\p)$ is a polynomial algebra over $k$, the degrees of a set of homogeneous generators are the eigenvalues plus one
of $\ad h$ on $\g_T$ (Lemma~\ref{eigenvaluesDe}) and the field $C(\p^-_{\pi'})$ of invariant fractions is a purely transcendental extension of $k$.
\end{theorem}

\section{Type ${\rm E}_7$}\label{E7}

 Let $\mathfrak g$ be of type ${\rm E}_7$ and let $\p=\p^-_{\pi',\,\Lambda}$ be the truncated maximal parabolic subalgebra corresponding to $\pi'=\pi\setminus \{\alpha_3\}$. Let $\beta_1$ be the unique highest root of $\g$ and let $H_{\beta_1}=\{\beta\in \Delta^+\,|\, (\beta,\beta_1)>0\}$ be the maximal Heisenberg set of centre $\beta_1$ in $\Delta^+$. Then notice that the set $\Delta\setminus (H_{\beta_1}\sqcup -H_{\beta_1})$ is a root system of type ${\rm D}_6$ and removing $\alpha_3$ corresponds to removing the extremal root from a system of type ${\rm D}_6$.

Write $(a_1,\, a_2,\, a_3, \,a_4,\, a_5, \,a_6, \,a_7)$ for the root $\sum_{i=1}^7a_i\alpha_i$ (with $a_i$ some integers).

The sets $S$ and $T$ given in Section~\ref{SecDe} for type ${\rm D}_6$ with $s=6$ lead us to taking for $S$ the set

$\begin{array}{lll}
S=\{\beta_1, (0,\,1,\,1,\,2,\,2,\,2,\,1),\, (0,\,1,\,1,\,1,\,1,\,0,\,0),\,(0,\,-1,\,0,\,-1,\,-1,\,0,\,0),\\
 (0,\,0,\,0,\,0,\,-1,\,-1,\,0),\,(0,\,0,\,0,\,0,\,0,\,0,\,-1)\}\end{array}$

and for $T$ the set

$\begin{array}{ll}
T=\{(-1,\,0,\,0,\,0,\,0,\,0,\,0),\, (0,\,0,\,0,\,1,\,1,\,0,\,0),\,(0,\,0,\,1,\,1,\,0,\,0,\,0),\\(0,\,-1,\,0,\,-1,\,-1,\,-1,\,-1),\,
(0,\,0,\,0,\,0,\,0,\,-1,\,0)\}.\end{array}$

More explicitly, we have added to the set $S$ in type ${\rm D}_6$ with $s=6$ (rewritten with respect to the roots in type ${\rm E}_7$) the highest root $\beta_1$, and to the set $T$ in type ${\rm D}_6$ with $s=6$ (rewritten with respect to the roots in type ${\rm E}_7$) we have added the negative root $-\alpha_1$.

For every $\gamma\in S\setminus\{\beta_1\}$, we take the same Heisenberg set $\Gamma_{\gamma}$ (rewritten with respect to the roots in type ${\rm E}_7$) as in type ${\rm D}_6$ with $s=6$ and we add the maximal Heisenberg set $H_{\beta_1}$. Observe that if $\alpha\in H_{\beta_1}$ and $\beta\in\Gamma_{\gamma}$ with $\gamma\in S\setminus\{\beta_1\}$ then one has that $\alpha+\beta\not\in S$.

Hence, by the extremal case in type ${\rm D}_6$ (see the remark~\ref{remcond}), it follows that all conditions of Lemma~\ref{refineregular} hold
for $y=\sum_{\gamma\in S}x_{\gamma}$. Then defining $h\in\h_{\Lambda}$ by $\gamma(h)=-1$ for all $\gamma\in S$,
one obtains that  $(h,\,y)$ is an adapted pair for $\p$.

Finally we show that $Y(\p)$ is polynomial. For this we need to calculate the $\langle {\bf ij}\rangle$-orbits in $\pi$ and the lower and improved upper bounds for $Y(\p)$.
The orbits are the $\Gamma_1=\{\alpha_1\}, \Gamma_2=\{\alpha_3\}, \Gamma_3=\{\alpha_2,\, \alpha_7\}, \Gamma_4=\{\alpha_4,\,\alpha_6\}$ and $\Gamma_5=\{\alpha_5\}$.
For the lower bound, we need to compute $\delta_\Gamma$ for all orbit $\Gamma$.

Let $\{\ep_i\}_{1\le i\le 8}$ be an orthonormal basis of $\mathbb R^8$ according to which the simple roots of $\g$ are expanded as in~\cite[Planche VI]{BOU}.

Recall that the fundamental weights $\varpi'_i$, $1\le i\le 7$, $i\neq 3$, are those for the Levi factor of $\p$.

A direct computation gives :

$\varpi_1'=\frac{1}{4}(\ep_1-\ep_2-\ep_3-\ep_4-\ep_5-\ep_6-\ep_7+\ep_8)$ and $\varpi'_1-\varpi_1=-\frac{1}{2}\varpi_3$,

$\varpi_2'=\frac{1}{6}(5\ep_1+\ep_2+\ep_3+\ep_4+\ep_5+\ep_6)$ and $\varpi'_2-\varpi_2=-\frac{2}{3}\varpi_3$,

$\varpi_4'=\frac{1}{3}(2\ep_1-2\ep_2+\ep_3+\ep_4+\ep_5+\ep_6)$ and $\varpi'_4-\varpi_4=-\frac{4}{3}\varpi_3$,

$\varpi_5'=\frac{1}{2}(\ep_1-\ep_2-\ep_3+\ep_4+\ep_5+\ep_6)$ and $\varpi'_5-\varpi_5=-\varpi_3$,

$\varpi_6'=\frac{1}{3}(\ep_1-\ep_2-\ep_3-\ep_4+2\ep_5+2\ep_6)$ and $\varpi'_6-\varpi_6=-\frac{2}{3}\varpi_3$,

$\varpi_7'=\frac{1}{6}(\ep_1-\ep_2-\ep_3-\ep_4-\ep_5+5\ep_6)$ and $\varpi'_7-\varpi_7=-\frac{1}{3}\varpi_3$.

Thus we get (recall proof of Lemma~\ref{calculboundsB}) :
$\delta_{\Gamma_1}=-2(\varpi_1-\varpi'_1)=-\varpi_3$. Similarly one gets $\delta_{\Gamma_2}=\delta_{\Gamma_3}=\delta_{\Gamma_5}=-2\varpi_3$ and $\delta_{\Gamma_4}=-4\varpi_3$.
Hence the lower bound is $(1-e^{-\varpi_3})^{-1}(1-e^{-2\varpi_3})^{-3}(1-e^{-4\varpi_3})^{-1}$.

Now for the improved upper bound, for each $\gamma\in T$ we will find $t(\gamma)\in \mathbb QS$ such that $\gamma+t(\gamma)$ is a multiple of $\varpi_3$. Denote by $s_i$ the $i$-th element of $S$ as it is written above.

For $\gamma=-\alpha_1$, we have that $t(\gamma)=2s_1$ and $\gamma+t(\gamma)=\varpi_3$.

For $\gamma=\alpha_4+\alpha_5$, we have $t(\gamma)=6s_1+3(s_2+s_3)+2(s_4+s_5)+s_6$ and $\gamma+t(\gamma)=4\varpi_3$.

For $\gamma=\alpha_3+\alpha_4$, we have $t(\gamma)=3s_1+s_2+s_3$ and $\gamma+t(\gamma)=2\varpi_3$.

For $\gamma=-(\alpha_2+\alpha_4+\alpha_5+\alpha_6+\alpha_7)$, we have $t(\gamma)=3s_1+2s_2+s_3+s_5$ and $\gamma+t(\gamma)=2\varpi_3$.

Finally, for $\gamma=-\alpha_6$, we have $t(\gamma)=3s_1+2s_2+s_3+s_4+s_5+s_6$ and $\gamma+t(\gamma)=2\varpi_3$.

We deduce that the lower bound coincides with the improved upper bound. Thus $Y(\p)$ is a polynomial algebra over $k$.

Then one checks that $h=-\alpha_1^\vee-\frac{13}{2}\alpha_2^\vee+3\alpha_4^\vee+\frac{11}{2}\alpha_5^\vee-2\alpha_6^\vee-\frac{1}{2}\alpha_7^\vee$. The eigenvalues of $\ad h$ on the elements of $\g_T$ are respectively: $2,\, 5,\, 7,\, 9,\,17$, hence the degrees of a set of homogeneous generators of $Y(\p)$ are $3,\, 6,\, 8,\, 10,\, 18$.

Thus we obtain the following Theorem.

\begin{theorem}\label{conE7}

Let $\p^-_{\pi'}$ be the maximal parabolic subalgebra of the simple Lie algebra $\g$
of type ${\rm E}_7$ corresponding to $\pi'=\pi\setminus\{\alpha_3\}$. Then the Poisson semicentre
$Sy(\p^-_{\pi'})$ is a polynomial algebra over $k$ in five homogeneous generators, having degrees
$3,\, 6,\, 8,\, 10,\, 18$ respectively, and there exists an affine slice $y+\g_T$ in $(\p^-_{\pi',\,\Lambda})^*$, which is also a Weierstrass section for $Y(\p^-_{\pi',\,\Lambda})$.

\end{theorem}

\section{Type ${\rm E}_6$.}\label{secE6}

Recall that the numbering of simple roots follows \cite[Planche V]{BOU}. 
In type ${\rm E}_6$ we know that the Poisson centre of the  truncated maximal parabolic subalgebra associated to $\pi'=\pi\setminus \{\alpha_s\}$ is polynomial for $s=3,\,4,\,5$ by~\cite{FJ2} (since both bounds $\ch\mathcal A$  and $\ch\mathcal B$ coincide), resp. for $s=2$ by~\cite{PPY} and an adapted pair was constructed in~\cite{FL}, resp.  in~\cite{J6bis}. It remains to examine the cases $s=1,\,6$, and by symmetry we may just assume that $s=6$. 
In the latter case, we have that $\p^-_{\pi',\,\Lambda}=\g'\ltimes\m$, where $\g'$ is the Levi factor of $\p^-_{\pi'}$ (of type ${\rm D}_5$), and $\m$ is the nilradical of $\p^-_{\pi'}$, which is an abelian ideal of $\g'$, isomorphic to the half-spin representation of $\s\o_{10}$. Moreover the group ${\rm Spin}_{10}$ acts on $\m$ with a dense open orbit, which has no divisors in the complement, and the stabiliser of an element in this orbit is $Q={\rm Spin}_7\ltimes {\rm exp}\,(k^8)$ (see~\cite[Summary Table]{PV}). By~\cite[Prop. 3.10]{Y1}, one has that the algebra of invariants $S(\q)^Q$ (with $\q={\rm Lie}\,Q$) is a polynomial ring in three generators and the general theory of~\cite{Y1} asserts that $Y(\p^-_{\pi',\,\Lambda})$ is also polynomial in the same number of generators (but the degrees were not known).

Here we give an adapted pair for $\p^-_{\pi',\,\Lambda}$ and show that for this pair, the improved upper bound $\mathcal B'$ coincides with the lower bound $\ch\mathcal A$. We also compute the degrees of the three generators of the polynomial algebra $Y(\p^-_{\pi',\,\Lambda})$.

%We used computer calculations in GAP~\cite{GAP} and found 3,662 adapted pairs. 

Recall some notations and hypotheses.
Consider $S\subset\Delta^+\sqcup\Delta^-_{\pi'}$, and for all $\gamma\in S$, let $\Gamma_{\gamma}\subset\Delta^+\sqcup\Delta^-_{\pi'}$ be a Heisenberg set.  Suppose that all the sets $\Gamma_{\gamma}$'s
are disjoint and set $\Gamma=\bigsqcup_{\gamma\in S}\Gamma_{\gamma}$. Set, for all $\gamma\in S$, $\Gamma_{\gamma}^0=\Gamma_{\gamma}\setminus\{\gamma\}$ and $O=\bigsqcup_{\gamma\in S}\Gamma_{\gamma}^0$. Set $\o=\g_{-O}$ (notation of Section~\ref{standardnotation}) and $y=\sum_{\gamma\in S}c_{\gamma}x_{\gamma}$ where all the $c_{\gamma}$'s are nonzero scalars. Denote by $\Phi_y$ the skew-symmetric bilinear form on $\g$ such that, for all $x,\,x'\in\g$, $\Phi_y(x,\,x')=K(y,\,[x,\,x'])$, where $K$ is the Killing form.

Instead of Lemma~\ref{refineregular}, we will use the following Lemma:

\begin{lemma}\label{regularE}

Assume further that

\begin{enumerate}
\item[(i)] There exist disjoint subsets $T^*$ and $T$ of $\Delta^+\sqcup\Delta^-_{\pi'}$, also disjoint from $\Gamma$, such that $\Delta^+\sqcup\Delta^-_{\pi'}=\Gamma\sqcup T^*\sqcup T$. 
\item[(ii)] The restriction of $\Phi_y$ to $\o\times\o$ is non-degenerate.
\item[(iii)] $S_{\mid\h_{\Lambda}}$ is a basis for $\h_{\Lambda}^*$.
\item[(iv)] For all $\beta\in T^*$, $x_{\beta}\in(\ad \p_{\pi',\,\Lambda}^-)\,y+\g_T$.
\item[(v)] $\lvert T\rvert=\ind\p_{\pi',\,\Lambda}$.
\end{enumerate}
Then $\p_{\pi',\,\Lambda}=(\ad \p_{\pi',\,\Lambda}^-)\,y\oplus\g_T$, where $\ad$ denotes the coadjoint action. In particular, $y$ is regular in $\p_{\pi',\,\Lambda}$. Moreover, if we uniquely define $h\in \h_{\Lambda}$ by the relations $\gamma(h)=-1$ for all $\gamma\in S$, then $(h,\,y)$ is an adapted pair for $\p_{\pi',\,\Lambda}^-$.
\end{lemma}

\begin{proof}
Condition (i) implies that $\p^-_{\pi',\,\Lambda}=\h_{\Lambda}\oplus\o\oplus\g_{-S}\oplus\g_{-T^*}\oplus\g_{-T}$ and that $\p_{\pi',\,\Lambda}=\h_{\Lambda}\oplus\g_O\oplus\g_{S}\oplus\g_{T^*}\oplus\g_{T}$.
Condition (ii) implies that $\g_O\subset(\ad\o)y+\g_S+\g_T+\g_{T^*}$ since $O\cap S=\emptyset$.
Condition (iii) implies that $\g_S=(\ad\h_{\Lambda})y$ and that $\h_{\Lambda}\subset(\ad\g_{-S})y+\g_O+\g_S+\g_T+\g_{T^*}$.
Condition (iv) implies that $\g_{T^*}\subset(\ad \p_{\pi',\,\Lambda}^-)\,y+\g_T$.
Hence $\p_{\pi',\,\Lambda}=\h_{\Lambda}\oplus\g_O\oplus\g_{S}\oplus\g_{T^*}\oplus\g_{T}\subset(\ad \p_{\pi',\,\Lambda}^-)\,y+\g_T$. Finally condition (v) implies that the latter sum is direct, since $\dim\g_T=\ind\p_{\pi',\,\Lambda}\le{\rm codim}\,(\ad \p_{\pi',\,\Lambda}^-)\,y$.
\end{proof}

Recall the strongly orthogonal positive roots $\beta_1$, $\beta_2$, $\beta_3$, $\beta_4$ of the Kostant cascade for $\Delta^+$ (see~\cite[Table I]{J1} or~\cite[Table I]{FL}) and $\beta'_1$, $\beta'_2$, $\beta'_3$ and $\beta'_{1'}$ for $\Delta^+_{\pi'}=-\Delta^-_{\pi'}$ (see Section~\ref{SecD}).

We choose for $S$ the set $S=\{\beta_1,\, \beta_2, \,\beta_3, \,-\beta_1', \,-\beta_2'+\alpha_2\}$ or in terms of simple roots, by writing as $(a_1,\,a_2,\,a_3,\,a_4,\,a_5,\,a_6)$ the root $\sum_{i=1}^6a_i\alpha_i$, our chosen set $S$ is the set
$\begin{array}{ll}
S=&\{(1,\,2,\,2,\,3,\,2,\,1),\, (1,\,0,\,1,\,1,\,1,\,1),\, (0,\,0,\,1,\,1,\,1,\,0), (-1,\,-1,\,-2,\,-2,\,-1,\,0),\\
& (0,\,0,\,0,\,-1,\,-1,\,0)\}.\end{array}$

We easily check  that $S_{\mid\h_{\Lambda}}$ is a basis for $\h_{\Lambda}^*$, hence condition (iii) of Lemma~\ref{regularE} is satisfied.

Set $\Gamma_{\beta_1}=H_{\beta_1}\setminus\{(0,\,1,\,1,\,1,\,0,\,0),\,(1,\,1,\,1,\,2,\,2,\,1)\}$, where $H_{\beta_1}$ is the maximal Heisenberg set in $\Delta^+$ defined in Example~\ref{MaxHeisenbergset}.

Set $\Gamma_{\beta_2}=H_{\beta_2}\setminus\{(1,\,0,\,1,\,1,\,1,\,0),\,(0,\,0,\,0,\,0,\,0,\,1)\}$ and $\Gamma_{\beta_3}=H_{\beta_3}$.
Set $\Gamma_{-\beta'_1}=-H_{\beta'_1}$ and $\Gamma_{-\beta'_2+\alpha_2}=\{-\beta'_2+\alpha_2,\,-\alpha_4,\,-\alpha_5\}$.
We easily check that all these sets are disjoint Heisenberg sets.

Now set $T^*=\{(1,\,1,\,1,\,2,\,2,\,1),\,(1,\,0,\,1,\,1,\,1,\,0),\,-\alpha_1,\,-\alpha_2,\,-(\alpha_2+\alpha_4),\,-(\alpha_2+\alpha_4+\alpha_5)\}$ and $T=\{\alpha_4,\, \alpha_6,\, \alpha_2+\alpha_3+\alpha_4\}$.

The $\langle {\bf ij}\rangle$-orbits in $\pi$ are $\Gamma_1:=\{\alpha_1,\, \alpha_6\}, \,\Gamma_2:=\{\alpha_2,\,\alpha_3,\,\alpha_5\}$ and $\Gamma_3:=\{\alpha_4\}$, hence condition (v) of Lemma~\ref{regularE} is satisfied.

By~\cite[Lemma 2.2]{J1} or~\cite[Lemma 3 (2)]{FL}, one has that $\Delta^+=\bigsqcup_{i=1}^4 H_{\beta_i}$. Moreover $H_{\beta_4}=\{\beta_4=\alpha_4\}$. Hence one has that $\Delta^+=\Gamma_{\beta_1}\sqcup\Gamma_{\beta_2}\sqcup\Gamma_{\beta_3}\sqcup(T^*\cap\Delta^+)\sqcup T$.

Similarly one has that $\Delta^+_{\pi'}=H_{\beta'_1}\sqcup H_{\beta'_2}\sqcup H_{\beta'_3}\sqcup H_{\beta'_{1'}}$.

Moreover $H_{\beta'_2}=\{\beta'_2,\,\alpha_2,\,\alpha_4+\alpha_5,\,\alpha_2+\alpha_4,\,\alpha_5\}$, $H_{\beta'_3}=\{\beta'_3=\alpha_4\}$ and $H_{\beta'_{1'}}=\{\beta'_{1'}=\alpha_1\}$.
Hence one has that $\Delta^-_{\pi'}=\Gamma_{-\beta'_1}\sqcup\Gamma_{-\beta'_2+\alpha_2}\sqcup(T^*\cap\Delta^-_{\pi'})$. Thus condition (i) of Lemma~\ref{regularE} is satisfied.

To prove condition (ii), it suffices to prove Lemma~\ref{non-degeneracy}, noting that $S^+=\{\beta_1,\,\beta_2,\,\beta_3\}$, $S^-=\{-\beta'_1,\,-\beta'_2+\alpha_2\}$ and $S^m=\emptyset$.
Using~\cite[Lemma 3 (5)]{FL}, condition (2) and (3) of Lemma~\ref{non-degeneracy} follow directly
and condition (4) is empty. Hence condition (ii).

It remains to prove condition (iv). By rescaling the nonzero root vectors $x_{\alpha}$, $\alpha\in\Delta$, one has that 

$x_{(1,\,1,\,1,\,2,\,2,\,1)}=(\ad(x_{(0,\,-1,\,-1,\,-1,\,0,\,0)}+x_{(1,\,0,\,1,\,0,\,0,\,0)}))y$; 

$x_{(1,\,0,\,1,\,1,\,1,\,0)}=(\ad x_{\alpha_1})y$;

$x_{-\alpha_1}=(\ad x_{(-1,\,0,\,-1,\,-1,\,-1,\,0)})y+x_{\alpha_6}$; 

$x_{-(\alpha_2+\alpha_4+\alpha_5)}=(\ad x_{(-1,\,-1,\,-1,\,-2,\,-2,\,-1)})y+x_{\alpha_2+\alpha_3+\alpha_4}$. 

Hence the roots $(1,\,1,\,1,\,2,\,2,\,1),\,(1,\,0,\,1,\,1,\,1,\,0),\,-\alpha_1,\,-(\alpha_2+\alpha_4+\alpha_5)$ satisfy condition (iv).

Finally one has that

$(\ad x_{(0,\,-1,\,-1,\,-2,\,-1,\,0)})y=x_{(1,\,1,\,1,\,1,\,1,\,1)}+x_{(0,\,-1,\,0,\,-1,\,0,\,0)}$

$(\ad x_{(0,\,1,\,0,\,0,\,0,\,0)})y=x_{(1,\,1,\,1,\,1,\,1,\,1)}+x_{(0,\,1,\,1,\,1,\,1,\,0)}$

$(\ad x_{(-1,\,-1,\,-1,\,-2,\,-1,\,-1)})y=x_{(0,\,1,\,1,\,1,\,1,\,0)}+x_{(0,\,-1,\,0,\,-1,\,0,\,0)}$.

Hence the root $-(\alpha_2+\alpha_4)$ satisfies condition (iv).

Similarly one has that

$(\ad x_{(0,\,-1,\,-1,\,-1,\,-1,\,0)})y=x_{(1,\,1,\,1,\,2,\,1,\,1)}+x_{(0,\,-1,\,0,\,0,\,0,\,0)}$

$(\ad x_{(-1,\,-1,\,-1,\,-1,\,-1,\,-1)})y=x_{(0,\,1,\,1,\,2,\,1,\,0)}+x_{(0,\,-1,\,0,\,0,\,0,\,0)}$

$(\ad x_{(0,\,1,\,0,\,1,\,0,\,0)})y=x_{(1,\,1,\,1,\,2,\,1,\,1)}+x_{(0,\,1,\,1,\,2,\,1,\,0)}$.

Hence the root $-\alpha_2$ satisfies condition (iv).

All conditions of Lemma~\ref{regularE} are satisfied. Thus we obtain an adapted pair $(h,\,y)$ for $\p^-_{\pi',\,\Lambda}$.

Now we compute the lower bound $\ch\mathcal A$ and the improved upper bound $\mathcal B'$
for $\ch Y(\p^-_{\pi',\,\Lambda})$.
 Note that $\pi'$ is of type ${\rm D}_5$ but we need to pay attention at the numbering of simple roots, which is different from the usual for ${\rm D}_5$. Denote by $\{\ep_i\}_{i=1}^8$ an orthonormal basis of $\mathbb R^8$ according to which the roots of ${\rm E}_6$ are expanded as in~\cite[Planche V]{BOU}. Recall that the fundamental weights $\varpi_i'$, $i\in \{1, \dots, 5\}$, are those of the Levi factor of $\p$. We have:
$\varpi_1'=\frac{1}{2}(\ep_8-\ep_7-\ep_5-\ep_6)$ and $\varpi'_1-\varpi_1=-\frac{1}{2}\varpi_6$,

$\varpi_2'=\frac{1}{2}(\ep_1+\ep_2+\ep_3+\ep_4)-\frac{1}{4}(\ep_5+\ep_6+\ep_7-\ep_8)$ and $\varpi'_2-\varpi_2=-\frac{3}{4}\varpi_6$,

$\varpi_3'=\frac{1}{2}(-\ep_1+\ep_2+\ep_3+\ep_4-\ep_5-\ep_6-\ep_7+\ep_8)$ and $\varpi'_3-\varpi_3=-\varpi_6$,

$\varpi_4'=\ep_3+\ep_4-\frac{1}{2}(\ep_5+\ep_6+\ep_7-\ep_8)$ and $\varpi'_4-\varpi_4=-\frac{3}{2}\varpi_6$,

$\varpi_5'=\ep_4-\frac{1}{4}(\ep_5+\ep_6+\ep_7-\ep_8)$ and $\varpi'_5-\varpi_5=-\frac{5}{4}\varpi_6$.

We may compute $\delta_\Gamma$, for each orbit $\Gamma$.
We have $\delta_{\Gamma_1}=-2(\varpi_1+\varpi_6-\varpi_1')=-3\varpi_6$,

$\delta_{\Gamma_2}=-2(\varpi_2+\varpi_3+\varpi_5-\varpi_2'-\varpi_3'-\varpi_5')=-6\varpi_6$,

and $\delta_{\Gamma_3}=-2(\varpi_4-\varpi_4')=-3\varpi_6.$

Hence the lower bound for $\ch Y(\p)$ is equal to $(1-e^{-3\varpi_6})^{-2}(1-e^{-6\varpi_6})^{-1}\le \ch Y(\p)$.

We now compute the improved upper bound; recall that for every $\gamma\in T$ we need to compute the unique element $t(\gamma)\in \mathbb QS$ such that $\gamma+t(\gamma)$ is a multiple of $\varpi_6$.

For $\gamma=\alpha_4$, one has $t(\gamma)=5\beta_1+3\beta_2+3\beta_3+2(-\beta_2'+\alpha_2)+4(-\beta_1')$ and $\gamma+t(\gamma)=6\varpi_6$.

For $\gamma=\alpha_6$, one has $t(\gamma)=2\beta_1+\beta_2+\beta_3+(-\beta_1')$ and $\gamma+t(\gamma)=3\varpi_6$.
For $\gamma=\alpha_2+\alpha_3+\alpha_4$, one has $t(\gamma)=2\beta_1+2\beta_2+\beta_3+2(-\beta_1')$ and $\gamma+t(\gamma)=3\varpi_6$.

Hence the improved upper bound coincides with the lower bound and $Y(\p^-_{\pi',\,\Lambda})$ is a polynomial algebra over $k$. Note that the element $h\in\h_{\Lambda}$ such that $\gamma(h)=-1$ for all $\gamma\in S$ is $h=-2\alpha_1^\vee-\alpha_2^\vee+\alpha_3^\vee+6\alpha_4^\vee-5\alpha_5^\vee$. Then the eigenvalues of $\ad h$ on the elements of $\g_T$ are $5,\,7$ and $17$, hence the degrees of a set of homogeneous generators for $Y(\p)$ are $6,\, 8$ and $18$.

Now we can give the following Theorem.

\begin{theorem}\label{conE6}

Let $\p^-_{\pi'}$ be the maximal parabolic subalgebra of the simple Lie algebra $\g$
of type ${\rm E}_6$ corresponding to $\pi'=\pi\setminus\{\alpha_6\}$. Then the Poisson semicentre
$Sy(\p^-_{\pi'})$ is a polynomial algebra over $k$ in three homogeneous generators, having degrees
$6,\, 8$ and $18$ respectively, and there exists an affine slice $y+\g_T$ in $(\p^-_{\pi',\,\Lambda})^*$, which is also a Weierstrass section for $Y(\p^-_{\pi',\,\Lambda})$.

\end{theorem}

\end{document}